\pdfoutput=1
\RequirePackage{ifpdf}
\ifpdf 
\documentclass[pdftex]{sigma}
\else
\documentclass{sigma}
\fi

\numberwithin{equation}{section}

\newtheorem{Theorem}{Theorem}[section]
\newtheorem{Corollary}[Theorem]{Corollary}
\newtheorem{Lemma}[Theorem]{Lemma}
\newtheorem{Proposition}[Theorem]{Proposition}
 { \theoremstyle{definition}
\newtheorem{Remark}[Theorem]{Remark} }

\usepackage{nicefrac} 
\usepackage[all]{xy}
\usepackage{bbm}

\newcommand{\bsfrac}[2]{%
\scalebox{-1}[1]{\nicefrac{\scalebox{-1}[1]{$#1$}}{\scalebox{-1}[1]{$#2$}}}%
}

\begin{document}

\allowdisplaybreaks

\newcommand{\arXivNumber}{2003.12820}

\renewcommand{\thefootnote}{}

\renewcommand{\PaperNumber}{081}

\FirstPageHeading

\ShortArticleName{Integral Structure for Simple Singularities}

\ArticleName{Integral Structure for Simple Singularities\footnote{This paper is a~contribution to the Special Issue on Primitive Forms and Related Topics in honor of Kyoji Saito for his 77th birthday. The full collection is available at \href{https://www.emis.de/journals/SIGMA/Saito.html}{https://www.emis.de/journals/SIGMA/Saito.html}}}

\Author{Todor MILANOV and Chenghan ZHA}
\AuthorNameForHeading{T.~Milanov and C.~Zha}

\Address{Kavli IPMU (WPI), UTIAS, The University of Tokyo, Kashiwa, Chiba 277-8583, Japan}
\Email{\href{mailto:todor.milanov@ipmu.jp }{todor.milanov@ipmu.jp}, \href{mailto:chenghan.zha@ipmu.jp}{chenghan.zha@ipmu.jp}}

\ArticleDates{Received May 27, 2020, in final form August 09, 2020; Published online August 22, 2020}

\Abstract{We compute the image of the Milnor lattice of an ADE singularity under a period map. We also prove that the Milnor lattice
can be identified with an appropriate relative $K$-group defined through the Berglund--H\"ubsch dual of the corresponding singularity.}


\Classification{14D05; 32S30; 19L47}

\renewcommand{\thefootnote}{\arabic{footnote}}
\setcounter{footnote}{0}

\section{Introduction}\label{introduction}

Let $f\in \mathbb{C}[x_1,x_2,x_3]$ be a weighted homogeneous polynomial representing
the germ of a simple singularity of type~A,~D, or~E. Let $f^T\in
\mathbb{C}[x_1,x_2,x_3]$ be the corresponding Berglund--H\"ubsch dual of $f$
(see Section~\ref{sec:K-Mil}). Fan--Jarvis--Ruan proved in~\cite{MR3043578}, using also
results of Givental--Milanov~\cite{MR2103007} and Frenkel--Givental--Milanov~\cite{MR2734560}, that the generating function of Fan--Jarvis--Ruan--Witten (FJRW)
invariants of $f^T$ can be identified with a tau-function of a
specific Kac--Wakimoto hierarchy. The identification however involves
rescaling the dynamical variables of the Kac--Wakimoto hierarchy and
the precise values of the rescaling constants were left unknown. One
application of the results in this papers is to obtain explicit
formulas for the rescaling constants. Such an explicit
identification is needed if one is interested in constructing
a matrix model for the FJRW invariants of $f^T$, similar to
the Kontsevich's matrix model in \cite{MR1171758}. We are not going to
compute the rescaling coefficients in this paper. The computation is
straightforward and it should probably be done only when needed. Let
us try to explain instead why this small technical detail leads to a very
interesting problem in singularity theory.

Let us recall that for any singularity $f$ there is a natural way to
construct a semi-simple Frobenius structure on the space of miniversal
deformations of $f$ (see \cite{MR1924259}). The construction depends on the choice of a
{\em primitive form} in the sense of Saito \cite{Sa} and it
essentially coincides with what Saito called {\em flat structure}. On
the other hand, motivated by Gromov--Witten theory, Givental
introduced the notion of a {\em total descendent potential} for every
semi-simple Frobenius manifold (see \cite{MR1901075, MR1866444}). Givental
conjectured \cite{MR1866444} and Teleman proved \cite{MR2917177} that if the
Frobenius structure corresponding to the quantum cohomology of a
compact K\"ahler manifold $X$ is semi-simple, then his definition
coincides with the generating function of Gromov--Witten invariants of
$X$. Let us return to our settings, i.e., the case of a simple
weighted homogeneous singularity $f$ on 3 variables. The standard
holomorphic volume form
$dx_1\wedge dx_2\wedge dx_3$ is primitive. Therefore, following
Givental, we can define total descendent potential. The latter will be
called, the total descendent potential of $f$. Fan--Jarvis--Ruan
proved in \cite{MR3043578} that the generating function of FJRW invariants
of $f^T$ coincides with the {\em total descendant potential} of
$f$. Furthermore, Givental--Milanov \cite{MR2103007} and
Frenkel--Givental--Milanov
\cite{MR2734560} proved that the total descendant potential of $f$
is a tau-function of the principal Kac--Wakimoto
hierarchy of the same type A, D, or E as the singularity $f$. Finally,
the outcome of the above work is that the generating function of FJRW
invariants of $f^T$ is a tau-function of an appropriate Kac--Wakimoto
hierarchy. However, there is still a small gap in this statement. Namely,
while the state space of FJRW theory is identified
explicitly with the Milnor ring of the singularity (see \cite{MR3043578}),
the identification of the Milnor ring and the Cartan subalgebra of the
corresponding simple Lie algebra is given by a period map and it is not
explicit. In order to obtain an explicit identification, we need to
determine the image of the root lattice in the Milnor ring of the singularity. This is
exactly the problem that we want to solve in this paper.

\subsection{Simple singularities}
Let us give a precise statement of the problem that we want to solve.
Let $f(x_1,x_2,x_3)=g(x_1,x_2) + x_3^2$, where $g$ is one of the
polynomials listed in the following table:
\begin{center}
\begin{tabular}{c|ccccc}
 Type & $A_N$ & $D_N$ & $E_6$ & $E_7$ & $E_8$ \\
 \hline
 $g$&
 $x_1^{N+1}+ x_2^2 $ &
 $x_1^2x_2+x_2^{N-1}$ &
 $x_1^3+x_2^4$ & $x_1^3+x_1 x_2^3$ & $x_1^3+x_2^5$\tsep{3pt}
\end{tabular}
\end{center}
The polynomial $f$ represents the germ of a simple singularity at
$x=0$. Let
\[
H_f:=\mathbb{C}[x_1,x_2,x_3]/(f_{x_1},f_{x_2},f_{x_3})
\]
be the {\em Milnor ring} of $f$, where $f_{x_i}:=\tfrac{\partial f}{\partial
 x_i}$. Let us denote by $(\ ,\ )$ the {\em residue pairing} on $H_f$
corresponding to the standard volume form $\omega=dx_1\wedge
dx_2\wedge dx_3$, that is,
\begin{gather*}
(\phi_1(x),\phi_2(x)) : =
\operatorname{Res}_{x=0} \frac{
 \phi_1(x)\phi_2(x) \omega }{
 f_{x_1}f_{x_2}f_{x_3}} .
\end{gather*}
The hypersurfaces $V_\lambda=\big\{x\in \mathbb{C}^3\, |\, f(x)=\lambda\big\}$ for
$\lambda\neq 0$ are non-singular and their union has a~structure of a~smooth fibration on $\mathbb{C}\setminus{\{0\}}$ known as the {\em Milnor
 fibration}. Let us fix a reference point $\lambda=1$ and consider
the middle homology group $H_2(V_1;\mathbb{Z})$, known also as the {\em
 Milnor lattice}. Our interest is in the period vectors
$I^{(-1)}_\alpha(\lambda)\in H_f$ defined by
\begin{gather*}
\big(I^{(-1)}_\alpha(\lambda),\phi_i\big):= \frac{1}{2\pi}
\int_{\alpha_\lambda} \phi_i(x)\frac{\omega}{df},
\end{gather*}
where $\alpha\in H_2(V_1;\mathbb{C})$, $\phi_i(x)$ ($1\leq i\leq N$) is a set of
polynomials representing a basis of $H_f$,
$\alpha_\lambda\in H_2(V_\lambda;\mathbb{C})$ is obtained from $\alpha$ via a
parallel transport along some reference path, and $\tfrac{\omega}{df}$
is the so-called {\em Gelfand--Leray} form (see
\cite{MR966191}). Alternatively, we can view each period vector as a~multivalued analytic function $I_\alpha^{(-1)}\colon \mathbb{C}\setminus{\{0\}}\to H_f$.

Let us assign degree $c_i\in \mathbb{Q}_{>0}$ to $x_i$ ($1\leq i\leq 3$), such that, the
polynomial $f$ has degree 1. Then the Milnor ring becomes a graded
ring. The highest possible degree of a homogeneous element in $H_f$ is $D=\sum\limits_{i=1}^3
(1-2c_i)= 1-\tfrac{2}{h}$, where $h$
is the Coxeter number of the corresponding root system. Put
$\theta:=\frac{D}{2}-\operatorname{deg}$, where
$\operatorname{deg}\colon H_f\to H_f$ is the linear operator uniquely
determined by the following condition: if $\phi$ is a weighted
homogeneous element of degree $d$, then $\deg(\phi)=d \phi$. For
homogeneity reasons, the period vectors have the form
\begin{gather}\label{psi-def}
I^{(-1)}_\alpha(\lambda) =
\frac{\lambda^{\theta+1/2} }{\Gamma(\theta+3/2)} \Psi(\alpha),
\end{gather}
where $\Psi\colon H_2(V_1;\mathbb{C})\to H_f$ is a linear isomorphism. Our goal is
to compute the image of the Milnor lattice $H_2(V_1;\mathbb{Z})$ via the map
$\Psi$. The solution to this problem is given in Section
\ref{sec:imML}. Explicit formulas for the image of the Milnor lattice via the map
$\Psi$ are given in Sections \ref{sec:a}--\ref{sec:e8}. The main
feature of our answer is that it involves various $\Gamma$-constants
and roots of unity. The second goal of our paper is to show that
although the formulas look cumbersome, in fact there is an interesting
structure behind them.

\subsection{K-theoretic interpretation of the Milnor lattice}
\label{sec:K-Mil}

It turns out that our answer can be stated quite elegantly via relative
K-theory. The idea to look for such a description comes from the
work of Iritani \cite{MR2553377}, Chiodo--Iritani--Ruan \cite{MR3210178}, and
Chiodo--Nagel \cite{MR3888782}. More precisely, Iritani was able to prove in \cite{MR2553377}
that the Milnor lattice of the mirror of a Fano toric orbifold $X$ can be
identified with the topological K-ring $K^0(X)$. The identification uses a period
map to embed the Milnor lattice in $H^*(X;\mathbb{C})$ and a certain $\Gamma$-class
modification of the Chern character map to embed $K^0(X)$ in
$H^*(X;\mathbb{C})$. The lattice in $H^*(X;\mathbb{C})$, obtained either as the image
of the Milnor lattice via the period map or as the image of $K^0(X)$
via the $\Gamma$-class
modification of the Chern character map, is known as $\Gamma$-integral
structure in quantum cohomology. Isolated singularities are almost never mirror models of
a manifold. Nevertheless, Chiodo--Iritani--Ruan have proposed an
analogue of the $\Gamma$-integral structure for singularities of
Fermat type. The analogue of $H^*(X;\mathbb{C})$ is played by the Milnor
ring~$H_f$, while $K^0(X)$ is replaced with an appropriate category of
equivariant matrix factorizations of $f$. Finally,
Chiodo--Nagel were able to find an isomorphism between $H_f$ and
an appropriate relative orbifold cohomology group. Since, the Chern
character gives an isomorphism between cohomology and K-theory and
the Grothendieck group of the category of matrix factorization also has
the flavor of a topological K-ring, after expecting more carefully the
constructions in~\cite{MR3210178} and~\cite{MR3888782}, we see that there is a~natural
candidate for a $\Gamma$-integral structure for Fermat type singularities.
After several trial and errors we were able to find the
correct topological $K$-ring and the correct modification of the Chern
character map. Moreover our proposal makes sense not only for Fermat
type polynomials, but more generally for an arbitrary {\em invertible
 polynomial}. Nevertheless, let us return to our current settings of simple
singularities. We believe that our results can be generalized to all
invertible polynomials, but that would require some additional work.

The polynomials $f$ corresponding to a simple singularity are
invertible polynomials in the sense of \cite{MR1214325} (see also \cite{MR2801653}). Each polynomial is
uniquely determined by a $3\times 3$ matrix $A=(a_{ij})_{1\leq i,j\leq
 3}$ with non-negative integer coefficients, such that,
\begin{gather*}
f(x) = \sum_{i=1}^3 x_1^{a_{i1}} x_2^{a_{i2}} x_3^{a_{i3}}.
\end{gather*}
Following Fan--Jarvis--Ruan (see \cite{MR3043578}) we consider also the
Berglund--H\"ubsch dual polynomial
\begin{gather*}
f^T(x) = \sum_{i=1}^3 x_1^{a_{1i}} x_2^{a_{2i}} x_3^{a_{3i}}.
\end{gather*}
Let $G^T$ be the group of diagonal symmetries of $f^T$, that is,
\begin{gather*}
G^T:=\big\{t\in (\mathbb{C}^*)^3\, |\, t_1^{a_{1i}} t_2^{a_{2i}}
t_3^{a_{3i}}=1\ \forall\, i \big\}.
\end{gather*}
Let $a^{ij}$ $(1\leq i,j\leq 3)$ be the entries of the inverse matrix
$A^{-1}$. The group $G^T$ is generated by the following elements
\begin{gather*}
\overline{\rho}_i = \big(e^{2\pi\mathbf{i} a^{i1}}, e^{2\pi\mathbf{i} a^{i2}},
e^{2\pi\mathbf{i} a^{i3}}\big), \qquad 1\leq i\leq 3.
\end{gather*}
Finally, let $V^T_1=\big\{x\in \mathbb{C}^3\, |\, f^T(x)=1\big\}$. Our main interest is
in the topological relative K-theoretic orbifold group
\begin{gather*}
K^0_{\rm orb}\big(\big[\mathbb{C}^3/G^T\big],\big[V^T_1/G^T\big]\big):=K^0_{G^T}\big(\mathbb{C}^3,V^T_1\big).
\end{gather*}
In general, there is no satisfactory definition of K-theory for
non-compact spaces. However, in our case the pair $\big(\mathbb{C}^3,V^T_1\big)$ is
$G^T$-equivariantly homotopic to a pair of finite $CW$ complexes, so
we may think of $\big(\mathbb{C}^3,V^T_1\big)$ as a $G^T$-equivariant pair of finite
$CW$-complexes. We refer to~\cite{MR234452} for some background on
equivariant topological K-theory.

Motivated by Iritani's $\Gamma$-integral structure in quantum cohomology (see \cite{MR2553377}), we
will now construct a linear map
\begin{gather}\label{ch-Gamma}
\operatorname{ch}_\Gamma\colon \
\xymatrix{
K^0_{\rm orb}\big(\big[\mathbb{C}^3/G^T\big],\big[V^T_1/G^T\big]\big)\otimes \mathbb{C} \ar[r] &
H_{\rm orb}\big(\big[\mathbb{C}^3/G^T\big],\big[V^T_1/G^T\big];\mathbb{C}\big),}
\end{gather}
which is a certain $\Gamma$-class modification of the orbifold Chern character map.
For a $G^T$-equivariant space $X$ and $g\in G^T$, let us denote
by $\operatorname{Fix}_g(X):=\{x\in X\, |\, gx=x\}$ the set of fixed
points. The elements in the relative $K$-group will be identified with
isomorphism classes $[E\to F]$ of two-term complexes $\xymatrix{E\ar[r]^{d} & F}$ of
$G^T$-equivariant vector bundles, such that, the differential~$d$ is a
morphism of $G^T$-equivariant vector bundles and
$d|_{V^T_1}\colon E_{V_1^T}\to F_{V_1^T}$ is an
isomorphism. Note that for $g\in G^T$, the
restriction of a vector bundle $E|_{\operatorname{Fix}_g(\mathbb{C}^3)}$
decomposes as a direct sum of eigen-subbundles $E_\zeta$ and that the
restriction to $\operatorname{Fix}_g\big(\mathbb{C}^3\big) $ of every two
term complex $\xymatrix{E\ar[r]^{d} & F}$ decomposes as a direct sum
of two term subcomplexes
$\xymatrix{E_\zeta\ar[r]^{d_\zeta} & F_\zeta}$, where
$d_\zeta=d|_{E_\zeta}$. We have the following well known
decomposition (e.g., see \cite[Theorem 2]{MR1076708}):
\begin{gather*}
\operatorname{Tr}\colon \
\xymatrix{
 K^0_{G^T}\big(\mathbb{C}^3,V^T_1\big)\otimes \mathbb{C} \ar[r]^-{\cong} &
 \bigoplus_{g\in G^T}
\big[ K^0\big(\operatorname{Fix}_g\big(\mathbb{C}^3\big) , \operatorname{Fix}_g\big(V^T_1\big)\big)
\otimes \mathbb{C} \big]^{G^T}, }
\end{gather*}
where $[\ ]^{G^T}$ denotes the $G^T$-invariant part and the morphism
$\operatorname{Tr}$ is defined by
\begin{gather*}
\operatorname{Tr}([E\to F]) =\bigoplus_{g\in G^T} \bigoplus_{\zeta\in \mathbb{C}^*}
\zeta [E_\zeta\to F_\zeta].
\end{gather*}
\begin{Remark}
The above decomposition is proved in~\cite{MR1076708} in the case of absolute
K-theory. However, using the long exact sequence of a pair, it is
straightforward to extend the result to relative K-theory as well.
\end{Remark}
The standard Chern character map gives an isomorphism
\begin{gather*}
\operatorname{ch} \colon \
\xymatrix{
K^0\big(\operatorname{Fix}_g\big(\mathbb{C}^3\big) , \operatorname{Fix}_g\big(V^T_1\big)\big) \otimes
\mathbb{C}
\ar[r] &
H^{\rm ev}\big(\operatorname{Fix}_g\big(\mathbb{C}^3\big) , \operatorname{Fix}_g\big(V^T_1\big);\mathbb{C}\big).}
\end{gather*}
Finally, if $G$ is a finite group acting on a smooth manifold $M$,
such that the quotient groupoid $[M/G]$ is an effective orbifold,
then $H^*(M/G;\mathbb{C})\cong [H^*(M;\mathbb{C})]^G$. Indeed, for a finite group
$G$ the operation taking $G$-invariants
is an exact functor from the category of $G$-vector spaces to the
category of vector spaces. Therefore
\begin{gather*}
H^i(M/G;\mathbb{C})\cong H^i\big([\Gamma(M,\mathcal{A}^*_M)]^G\big) =
[H^i(M,\mathcal{A}^*_M)]^G \cong [H^i(M;\mathbb{C})]^G,
\end{gather*}
where $\mathcal{A}^*_M$ is the sheaf of smooth differential forms on $M$ with
complex coefficients, the first isomorphism is Satake's de Rham theorem
for orbifolds (see~\cite{MR79769}), and the last one is the de Rham's
theorem for the manifold $M$. Using the long exact sequence of a pair,
we get also that $H^i(M/G,N/G;\mathbb{C})\cong [H^i(M,N;\mathbb{C})]^G$ for any
$G$-invariant submanifold $N\subset M$. On the other hand, by definition,
\begin{gather*}
H^*_{\rm orb} \big(\big[\mathbb{C}^3/G^T\big],\big[V_1^T/G^T\big];k\big) = \bigoplus_{g\in G^T}\!
H^*\big(\operatorname{Fix}_g\big(\mathbb{C}^3\big)/G^T,
\operatorname{Fix}_g\big(V_1^T\big)/G^T;k\big),
\qquad\!\!\!\! k=\mathbb{Q}, \mathbb{R}, \mathbb{C}.
\end{gather*}
Therefore, the composition
$\widetilde{\operatorname{ch}}:=\operatorname{ch}\circ
\operatorname{Tr}$ defines a ring homomorphism
\begin{gather*}
\widetilde{\operatorname{ch}}\colon \
\xymatrix{
 K^0_{\rm orb} \big(\big[\mathbb{C}^3/G^T\big],\big[V_1^T/G^T\big]\big)\otimes \mathbb{C} \ar[r] &
 H^{\rm ev}_{\rm orb} \big(\big[\mathbb{C}^3/G^T\big],\big[V_1^T/G^T\big];\mathbb{C}\big), }
\end{gather*}
which is the orbifold version of the Chern character map. Clearly
$\widetilde{\operatorname{ch}}$ is an isomorphism over $\mathbb{C}$.
\begin{Remark}
Orbifold cohomology $H^*_{\rm orb}$ has two natural gradings --
standard topological degree grading coming from the topological space underlying the orbit
space and Chen--Ruan grading. In this paper we work with the
topological grading and the topological cup product.
\end{Remark}
Let us recall also the definition of the $\Gamma$-class. If $E\in
K^0_{\rm orb}\big(\big[\mathbb{C}^3/G^T\big]\big):=K^0_{G^T}\big(\mathbb{C}^3\big)$ is an orbifold vector
bundle and $\operatorname{Tr}(E) = \sum_g\sum_{\zeta} \zeta E_\zeta$,
then each eigenvalue $\zeta=e^{2\pi\mathbf{i} \alpha}$, where $0\leq \alpha<1$ is a
rational number and we define
\begin{gather*}
\widehat{\Gamma}(E) = \sum_g \prod_{\zeta=e^{2\pi\mathbf{i}\alpha}}
\prod_{i=1}^{\operatorname{rk}(E_\zeta)}
\Gamma(1-\alpha + \delta_{\zeta,i})
 \in
H^{\rm ev}_{\rm orb} \big(\big[\mathbb{C}^3/G^T\big]\big),
\end{gather*}
where $\delta_{\zeta,i}$ ($1\leq i\leq \operatorname{rk}(E_\zeta)$)
are the Chern roots of the vector bundle $E_\zeta$. If
$E=\big[T\mathbb{C}^3/G^T\big]$ is the orbifold tangent bundle, then the
$\Gamma$-class is denoted by $\widehat{\Gamma}\big(\big[\mathbb{C}^3/G^T\big]\big) $.
The map~\eqref{ch-Gamma} is defined by the following formula:
\begin{gather*}
\operatorname{ch}_\Gamma([E\to F]):=
\frac{1}{2\pi} \widehat{\Gamma}\big(\big[\mathbb{C}^3/G^T\big]\big) \cup
(2\pi\mathbf{i})^{\operatorname{deg}_\mathbb{C}}
\iota^*\widetilde{\operatorname{ch}} ([E\to F]),
\end{gather*}
where
$\operatorname{deg}_\mathbb{C}(\phi)=i \phi$ for
$\phi\in H^{2i}_{\rm orb}\big(\big[\mathbb{C}^3/G^T\big],\big[V_1^T/G^T\big];\mathbb{C}\big)$ and $\iota^*$
is an involution in orbifold cohomology that exchanges the direct
summands corresponding to $g$ and $g^{-1}$. Note that the definition
of $\iota^*$ makes sense because
$\operatorname{Fix}_g=\operatorname{Fix}_{g^{-1}}$.
\begin{Theorem}\label{t1}
 There exists a linear isomorphism
 \begin{gather*}
 \operatorname{mir}\colon \ \xymatrix{H_f\ar[r] &
 H^*_{\rm orb} \big(\big[\mathbb{C}^3/G^T\big],\big[V_1^T/G^T\big];\mathbb{C}\big),}
 \end{gather*}
 such that, the map
 \begin{gather*}
 \operatorname{mir}^{-1}\circ \operatorname{ch}_\Gamma\colon \
 \xymatrix{K^0_{\rm orb}\big(\big[\mathbb{C}^3/G^T\big],\big[V_1^T/G^T\big]\big)\ar[r]^-{\cong} &
 \Psi\big(H_2\big(f^{-1}(1);\mathbb{Z}\big)\big)}
 \end{gather*}
 is an isomorphism of Abelian groups.
\end{Theorem}
Unfortunately we do not have a conceptual definition of the map
$\operatorname{mir}$. Our definition is on a case by case basis. We
expect that $H^*_{\rm orb} \big(\big[\mathbb{C}^3/G^T\big],\big[V_1^T/G^T\big];\mathbb{C}\big)$ has a~natural identification with the state space of FJRW-theory under which
$\operatorname{mir}$ is identified with the mirror map of
Fan--Jarvis--Ruan (see~\cite{MR3043578}). Let us point out also that in all cases the
following two properties are satisfied:
\begin{enumerate}\itemsep=0pt
\item
 If $x_1^{m_1}x_2^{m_2}x_3^{m_3}$ is a
 homogeneous monomial representing a vector in $H_f$, then its image
 under $\operatorname{mir}$ is in the twisted sector
 corresponding to
 $g=\overline{\rho}_1^{\, m_1+1}\overline{\rho}_2^{\,
 m_2+1}\overline{\rho}_3^{\, m_3+1}$.
\item
 The map $\operatorname{mir}$ is defined over $\mathbb{Q}$, that is,
 $\operatorname{mir}$ provides an isomorphism
 \begin{gather*}
 \mathbb{Q}[x_1,x_2,x_3]/(f_{x_1},f_{x_2},f_{x_3})\cong
 H^*_{\rm orb}\big(\big[\mathbb{C}^3/G^T\big],\big[V^T_1/G^T\big];\mathbb{Q}\big).
 \end{gather*}
\end{enumerate}

\section{Period map image of the Milnor lattice}\label{sec:imML}

\subsection{Suspension isomorphism in vanishing homology}

We will reduce the problem of computing periods of the hypersurface~$V_\lambda$ to computing periods of the Riemann surfaces
\begin{gather*}
M_\mu:=\big\{(x_1,x_2)\in \mathbb{C}^2\, |\, g(x_1,x_2)=\mu\big\}.
\end{gather*}
Consider the map $V_\lambda\to \mathbb{C}$, $(x_1,x_2,x_3)\mapsto
g(x_1,x_2)$. The fibers of this map are given by
\begin{gather*}
V_{\lambda,\mu} := M_\mu \times \big\{{-}\sqrt{\lambda-\mu}, \sqrt{\lambda-\mu}\big\}.
\end{gather*}
Suppose now that $A\in H_1(M_\lambda;\mathbb{Z})$ is any cycle. The
following two maps
\begin{gather*}
\phi_\pm\colon \ A\times [0,1] \to V_\lambda,\qquad
(x_1,x_2,t)\mapsto
\big(t^{c_1} x_1, t^{c_2} x_2, \pm \sqrt{\lambda (1-t)} \big)
\end{gather*}
have images that fit together and give a two-dimensional cycle
$\alpha\in V_\lambda$, that is,
$\alpha=\Sigma A$ is the suspension of the cycle $A$. It is
known that the above suspension operation $\Sigma\colon
H_1(M_\lambda;\mathbb{Z})\to H_2(V_\lambda;\mathbb{Z})$ is an isomorphism (see
\cite[Theorem 2.9]{MR966191}).

Note that we may choose the basis of $H_f$ to be such that
$\phi_i=\phi_i(x_1,x_2)$ does not depend on~$x_3$. Then the integral
\begin{gather*}
\frac{1}{2\pi}\int_{\alpha_\lambda} \phi_i \frac{\omega}{df}
=\frac{1}{2\pi} \partial_\lambda \int_{\alpha_\lambda} d^{-1}(\phi_i
\omega) =
\frac{1}{2\pi} \partial_\lambda \int_{\alpha_\lambda}
x_3\phi_i(x_1,x_2) dx_1\wedge dx_2,
\end{gather*}
where in the first equality we used the Stoke's theorem (see
\cite[Lemma~7.2]{MR966191}).
Using Fubini's theorem (see \cite[Lemma~7.2]{MR966191}), we have
\begin{gather*}
\int_{\alpha_\lambda} x_3\phi_i(x_1,x_2) dx_1\wedge dx_2 =
\int_0^\lambda (\lambda-\mu)^{1/2} \int_{A_\mu} \frac{\phi_i(x_1,x_2)
 dx_1 dx_2}{dg} d\mu \\
\hphantom{\int_{\alpha_\lambda} x_3\phi_i(x_1,x_2) dx_1\wedge dx_2 =}{}
-\int_\lambda^0 (-(\lambda-\mu)^{1/2}) \int_{A_\mu} \frac{\phi_i(x_1,x_2)
 dx_1 dx_2}{dg} d\mu,
\end{gather*}
where the first integral represents integrating over
$\phi_+(A\times[0,1])$, the second one over
$\phi_-(A\times[0,1])$, and $A_\mu\in H_1(M_\mu)$ for
$\mu=\lambda t$ is obtained from $A$ via the rescaling
$(x_1,x_2)\mapsto \big(t^{c_1}x_1, t^{c_2}x_2\big)$. We get
\begin{gather}\label{susp-period}
\frac{1}{2\pi}\int_{\alpha_\lambda} \phi_i \frac{\omega}{df}=
\frac{1}{\pi} \partial_\lambda
\int_0^\lambda (\lambda-\mu)^{1/2} \int_{A_\mu} \frac{\phi_i(x_1,x_2)
 dx_1 dx_2}{dg} d\mu.
\end{gather}
The image of the Milnor lattice $H_2(V_1;\mathbb{Z})$ will be computed with
formula \eqref{susp-period}.

\subsection{Simple singularities and root systems}
Let us first recall several well known
facts about simple singularities, which will be needed in our
computation (see \cite[Theorem~3.14]{MR966191}). The analytic continuation of
$I^{(-1)}_\alpha(\lambda)$ along a loop around $\lambda=0$ yields
$I^{(-1)}_{\sigma(\alpha)}(\lambda)$, where $\sigma\colon H_2(V_1;\mathbb{Z})\to
H_2(V_1;\mathbb{Z})$ is the so-called {\em classical monodromy}
operator. Recalling the definition of $\Psi$ (see formula~\eqref{psi-def}), we get the following relation:
\begin{gather}\label{Psi of classical monodromy operator}
\Psi(\sigma(\alpha)) = - e^{2\pi \mathbf{i} \theta} \Psi(\alpha),
\end{gather}
where $\mathbf{i}:=\sqrt{-1}$. In particular, knowing the image of one cycle $\alpha$ allows us to
find the image of the entire $\sigma$-orbit of $\alpha$.

Let us define
\begin{gather*}
(\alpha|\beta):= \lambda \big(I^{(0)}_\alpha(\lambda),I^{(0)}_\beta(\lambda)\big),
\end{gather*}
where $I^{(0)}_\alpha(\lambda):=\partial_\lambda
I^{(-1)}_\alpha(\lambda)$. It is straightforward to check that
\begin{gather}\label{ip-formula}
(\alpha|\beta)= \frac{1}{\pi} (\Psi(\alpha),\cos (\pi \theta) \Psi(\beta)).
\end{gather}
It is known that $(\alpha|\beta)=-\alpha\circ \beta$, where $\circ $
is the intersection pairing (see \cite{MR2103007,Sa}). In particular, the
form $(\ |\ )$ takes integer values on the Milnor lattice.

Finally, let us also recall that we have the following remarkable
facts (see \cite[Theorem 3.14]{MR966191}):
\begin{enumerate}\itemsep=0pt
\item The set of vanishing cycles of the singularity $f$ coincides
with the set of all $\alpha\in H_2(V_1;\mathbb{Z})$ such that
$(\alpha|\alpha)=2$.
\item The triple (Milnor lattice, set of vanishing cycles, pairing
$(\ |\ )$) form a root system of the
same type as the type of the singularity $f$, that is, the set of vanishing
cycles corresponds to the roots, the Milnor lattice corresponds to the
root lattice, and $(\ |\ )$ corresponds to the invariant bilinear
form.
\item The classical monodromy corresponds to a Coxeter transformation.
\end{enumerate}

\subsection[$A_N$-singularity]{$\boldsymbol{A_N}$-singularity}\label{sec:a}

Let us fix the following basis of $H_f$:
\[ \phi_i=x_1^{i-1}, \qquad 1\leqslant i\leqslant N.\]
The residue pairing takes the form
\[(\phi_i,\phi_j)=\frac1{4h}\delta_{i+j,h}, \qquad 1\leqslant i,j\leqslant N, \]
where $h=N+1$ is the Coxeter number. The Riemann surface $M_\mu$
for $\mu\neq 0$ is a non-singular curve in $\mathbb{C}^2$ defined by the equation
$x_1^{N+1}+x_2^2=\mu$. The projection $(x_1,x_2)\mapsto x_1$
defines a~degree~2 branched covering $M_\mu \to \mathbb{C}$, with
branching points $x_{1,k}=\mu^{\tfrac{1}{N+1}} \eta_{N+1}^k$
($k\in\mathbb Z_{N+1}$), where $\eta_{N+1}:=e^{2\pi\mathbf{i} /(N+1)}$ and
$\mathbb{Z}_{N+1}:=\mathbb{Z}/(N+1)\mathbb{Z}$.

Let us construct a basis of $H_1(M_\mu;\mathbb{Z})\cong \mathbb{Z}^N$. Cycles on
$M_\mu$ can be visualized easily via their projections on the
$x_1$-plane $\mathbb{C}$. Let $L_k$ ($k\in\mathbb Z_{N+1}$) be the
line segment $[0,x_{1,k}] $ (in the
$x_1$-plane).
Let $A^\prime_k$ be a loop in the $x_1$-plane that starts at $x_1=0$,
it goes along the line segment $L_k$, just before hitting the branched
point $x_{1,k}$ it makes a small loop $C_k$ counterclockwise around
$x_{1,k}$, it returns back to the starting point along the line
segment $L_k$, and then the loop continues to travel in a similar
fashion along $L_{k+1}$ except that this time we make a small loop
$C_{k+1}^{-1}$ clockwise around $x_{1,k+1}$. In other words
$A_{k}^\prime= L_{k+1}^{-1}\circ C_{k+1}^{-1}\circ L_{k+1}\circ
L_k^{-1}\circ C_k\circ L_k$. Note that $A_{k}^\prime$ lifts to two
loops $A_{k,a}$, $a\in\mathbb Z_2$ on $M_{\mu}$, where the starting
point of $A^\prime_k$ lifts to
$x_{2,k,a}:=\mu^\frac12(-1)^{a}$. The cycles $A_{k,a}$ satisfy the
following relations $A_{k,0}=-A_{k,1}$ and
$\sum\limits_{k=0}^NA_{k,a}=0$. Let us asume $a\in\mathbb Z_2\setminus\{0\}$
and $k\in\mathbb Z_{N+1}\setminus\{0\}$, then we get $N$ loops whose
homology classes, as we will see later on, represent a basis of $H_1(M_\mu;\mathbb{Z})$.

Let us compute the periods of the holomorphic forms
\[ \phi_i(x_1,x_2)\frac{dx_1dx_2}{dg}=-\frac{x_1^{i-1}dx_1}{2x_2} \]
along the cycles $A_{k,a}$.
The paths $L_k$ and $C_k$ can be parametrized as follows:
\begin{gather*}
L_k\colon \ x_1=\eta_{N+1}^k\mu^{\frac1{N+1}}t^\frac 1{N+1}, \qquad
0\leqslant t\leqslant\left(1-\frac{\epsilon}{\mu^\frac1{N+1}}\right)^{N+1},\\
C_k\colon \ x_1=\eta_{N+1}^k\mu^{\frac1{N+1}}+\epsilon e^{\mathbf{i}\theta},
 \qquad
\frac{2k-N-1}{N+1}\pi\leqslant \theta\leqslant \frac{2k+N+1}{N+1}\pi.
\end{gather*}
The integrals along the lifts of $C_k$ contribute to the period
integral terms of order $O\big(\epsilon^{\frac12}\big)$. These terms vanish in
the limit $\epsilon\to 0$. The periods that we want to compute are
independent of $\epsilon$ for homotopy reasons. Therefore, by passing
to the limit $\epsilon\to 0$ we get
\begin{gather*}
\int_{A_{k,a}}\phi_i(x_1,x_2)\frac{dx_1dx_2}{dg} =(1-(-1))\left(\int_{L_{k,a}}-\int_{L_{k+1,a}}\right)\frac{-x_1^{i-1}dx_1}{2x_2}\\
\hphantom{\int_{A_{k,a}}\phi_i(x_1,x_2)\frac{dx_1dx_2}{dg}}{}
 =\left(\int_{L_{k+1,a}}-\int_{L_{k,a}}\right)\frac{x_1^{i-1}dx_1}{x_2} .
\end{gather*}

The integrals along $L_{k,a}$ can be expressed in terms of Euler's Beta function
$B(a,b):=\tfrac{\Gamma(a)\Gamma(b)}{\Gamma(a+b)}$,
\begin{gather*}
\int_{L_{k,a}}\frac{x_1^{i-1}dx_1}{x_2}=
(-1)^a\int_0^1\frac{
 \eta_{N+1}^{ki}\mu^{\frac i{N+1}}t^{\frac i{N+1}-1}dt }{
 (N+1)\mu^\frac12(1-t)^\frac12}=
(-1)^a\frac{
 \eta_{N+1}^{ki}\mu^{\frac i{N+1}-\frac12} }{ N+1}
B\left(\frac i{N+1},\frac12\right).
\end{gather*}
Let $\alpha_{k,a}=\Sigma A_{k,a}$ be the suspension. Recalling formula
\eqref{susp-period} and using that
\begin{gather}\label{integral of mu to get beta function}
\int_0^\lambda (\lambda-\mu)^{1/2} \mu^a d\mu = \lambda^{a+3/2} B(a+1,3/2),
\end{gather}
we get
\begin{gather*} \big(I^{(-1)}_{\alpha_{k,a}}(\lambda),\phi_{i}\big):= \frac1{2\pi}\int_{\alpha_{k,a}}\phi_i\frac{\omega}{df}=\frac1\pi\partial_\lambda\int_0^\lambda(\lambda-\mu)^\frac12\int_{A_{k,a}}\phi_i(x_1,x_2)\frac{dx_1dx_2}{dg}d\mu\\
\hphantom{\big(I^{(-1)}_{\alpha_{k,a}}(\lambda),\phi_{i}\big)}{} = (-1)^a\frac{\eta_{N+1}^{ki}\big(\eta_{N+1}^{i}-1\big)}{2i}\lambda^{\frac i{N+1}}.
\end{gather*}
Recalling the formulas for the residue pairing we get
\begin{gather*} I^{(-1)}_{\alpha_{k,a}}(\lambda)=4h\sum_{i=1}^N(I^{(-1)}_{\alpha_{k,a}}(\lambda),\phi_{h-i})\phi_i. \end{gather*}
By definition $\theta(\phi_i)=\big(\frac12-\frac i{N+1}\big)\phi_i$.
Therefore, using \eqref{psi-def}, we get
\begin{gather}\label{A-psi}
\Psi(\alpha_{k,a})=(-1)^a2\sum_{i=1}^N\eta_{N+1}^{-ki}\big(\eta_{N+1}^{-i}-1\big)
\Gamma\left(1-\frac i{N+1}\right)\phi_i.
\end{gather}
Let us point out that formula \eqref{Psi of classical monodromy
 operator} yields the following formulas for the classical monodromy operator
\begin{gather*}
\sigma(\alpha_{k,a})=-\alpha_{k+1,a+1}=\alpha_{k+1,a},\qquad
k\in\mathbb Z_{N+1}, \qquad a\in\mathbb Z_2. \end{gather*}
The intersection pairing takes the form
\begin{gather*}(\alpha_{k,1}|\alpha_{l,1})= \frac1\pi(\Psi(\alpha_{k,1}),\cos(\pi\theta)\Psi(\alpha_{l,1}))\\
\hphantom{(\alpha_{k,1}|\alpha_{l,1})}{}
= \frac2h\sum_{i=1}^N\eta_{N+1}^{(l-k)i}\left(1-\cos\left(\frac{2i\pi}h\right)\right)=2\delta_{k,l}-\delta_{l-k,1}-\delta_{l-k,N},
\end{gather*}
where $k,l\in\mathbb Z_{N+1}$ and the Kronecker delta is also on
$\mathbb Z_{N+1}$. Note that $(\alpha_{k,1}|\alpha_{k,1})=2$, so
$\alpha_{k,1}$ is a vanishing cycle. The determinant of the
intersection pairing in the basis $\{\alpha_{k,1}\}$, that is, the
determinant of the matrix $(\alpha_{k,1}|\alpha_{l,1})_{k,l=1}^N$ is
$N+1$, which coincides with the determinant of the Cartan matrix of the
simple Lie algebra of type $A_N$. Therefore, $\{\alpha_{k,1}\}$ is a
$\mathbb{Z}$-basis of the root lattice, that is, $H_2(V_\lambda;\mathbb{Z})$ and hence
their images $\Psi(\alpha_{k,1})$ (see formula \eqref{A-psi}) give a
basis for the image of the Milnor lattice in $H_f$.

\subsection[$D_N$-singularity]{$\boldsymbol{D_N}$-singularity}

Let us fix the following basis of $H_f$:
\begin{gather*}
\phi_i(x_1,x_2) =
\begin{cases}
x_2^{i-1} & \mbox{if } 1\leq i\leq N-1, \\
2x_1 & \mbox{if } i=N.
\end{cases}
\end{gather*}
The residue pairing takes the form
\begin{gather*}
(\phi_i,\phi_j)=\frac{1}{2h} \delta_{i+j,N},\quad
1\leq i,j\leq N-1,\qquad
(\phi_i,\phi_N)=-\delta_{i,N}, \quad
1\leq i\leq N,
\end{gather*}
where $h=2N-2$ is the Coxeter number. The Riemann surface $M_\mu$
for $\mu\neq 0$ is a non-singular curve in $\mathbb{C}^2$ defined by the equation
$x_1^2x_2+x_2^{N-1}=\mu$. The projection $(x_1,x_2)\mapsto x_2$
defines a~degree~2 branched covering $M_\mu \to \mathbb{C}^*$, with
branching points $x_{2,k}=\mu^{\tfrac{1}{N-1}} \eta^{2k}$ ($1\leq k\leq
N-1$), where $\eta=e^{2\pi\mathbf{i} /h}$. Let $A'_k$ be a simple loop in
$\mathbb{C}^*$ around the line segment $L_k:=[0, x_{2,k}]$, that is, $A_k'$
is a loop starting at a point on the line segment $L_k$
sufficiently close to~$0$, it goes along the line segment $L_k$, just
before hitting the branch point $x_{2,k}$ it makes a small loop $C_k$ around it and it
returns back to the starting point along the line segment $L_k$, and
finally it makes a small loop $C_0$ around $0$. Clearly, the loop
$A_k'=C_0\circ L_k^{-1} \circ C_k\circ L_k$ lifts to a loop $A_k$ in
$M_\mu$. Let us compute the periods of the holomorphic forms
\begin{gather*}
\phi_i(x_1,x_2) \frac{dx_1 dx_2}{dg} =
\begin{cases}
\dfrac{x_2^{i-1} dx_2}{2x_1 x_2} & \mbox{if } 1\leq i\leq N-1 , \vspace{1mm}\\
\dfrac{dx_2}{x_2} & \mbox{if } i=N,
\end{cases}
\end{gather*}
along the cycle $A_k$. If $i=N$, then the period integral is just
$2\pi \mathbf{i}$. Suppose that $1\leq i\leq N-1$. Let us parametrize $A'_k$
as follows:
\begin{gather*}
\nonumber
C_0\colon \ x_2=\epsilon e^{\mathbf{i} \theta},\qquad
0\leq \theta \leq 2\pi, \\
\nonumber
L_k\colon \ x_2 = \mu^{\tfrac{1}{N-1}} \eta^{2k} t, \qquad
\epsilon \mu^{-\tfrac{1}{N-1}} \leq
t\leq 1-\epsilon \mu^{-\tfrac{1}{N-1}} , \\
\nonumber
C_k\colon \ x_2 = \mu^{\tfrac{1}{N-1}} \eta^{2k} + \epsilon e^{\mathbf{i}
 \theta},\qquad
0\leq \theta \leq 2\pi.
\end{gather*}
The integrals along the lifts of $C_0$ and $C_k$ contribute to the
period integral terms of orders respectively $O\big(\epsilon^{i-1/2}\big)$ and
$O\big(\epsilon^{1/2}\big)$. These terms vanish in the limit $\epsilon \to 0$.
The two lifts of $L_k$, before and after going around the branch point
$x_{2,k}$, have parametrizations, such that,
\begin{gather*}
x_2= \mu^{\tfrac{1}{N-1}} \eta^{2k} t,\qquad
x_1 x_2= \sqrt{\big(\mu-x_2^{N-1}\big)x_2 } =
\eta^{k} \mu^{1/2+1/h} \big(1-t^{N-1}\big)^{1/2} t^{1/2},
\end{gather*}
where $t$ varies from $0$ to $1$, and
\begin{gather*}
x_2= \mu^{\tfrac{1}{N-1}} \eta^{2k} t,\qquad
x_1 x_2= -\sqrt{\big(\mu-x_2^{N-1}\big)x_2 } =
-\eta^{k} \mu^{1/2+1/h} \big(1-t^{N-1}\big)^{1/2} t^{1/2},
\end{gather*}
where $t$ varies from $1$ to $0$. Now it is clear that the period
integral, after passing to the limit $\epsilon \to 0$, takes the form
\begin{gather*}
\int_{A_k} \phi_i \frac{dx_1 dx_2}{dg} =
\mu^{\tfrac{m_i}{h} -\tfrac{1}{2}} \eta^{m_i k}
\int_0^1 t^{i-\tfrac{3}{2} } \big(1-t^{N-1}\big)^{-1/2} dt,
\end{gather*}
where $m_i:=2i-1$ ($1\leq i\leq N-1$).
The above integral can be computed as follows,
\begin{gather*}
\int_0^1 t^{i-3/2} \big(1-t^{N-1}\big)^{-1/2} dt = \frac{1}{N-1} \int_0^1
s^{(2i-1)\tfrac{1}{h}-1} (1-s)^{-\tfrac{1}{2}} ds =
 \frac{1}{N-1} B\left(\frac{m_i}{h},\frac{1}{2}\right).
\end{gather*}
We get the following formulas:
\begin{gather*}
\int_{A_k} \phi_i \frac{dx_1 dx_2}{dg} =
\begin{cases}
\dfrac{1}{N-1}
\mu^{\tfrac{m_i}{h} -\tfrac{1}{2}}
\eta^{m_i k}
B\left(\dfrac{m_i}{h},\dfrac{1}{2}\right) ,& \mbox{ if } 1\leq i\leq N-1,\\
2\pi \mathbf{i}, & \mbox{ if } i=N.
\end{cases}
\end{gather*}
Let $\alpha_k=\Sigma A_k$ be the suspension. Recalling formula
\eqref{susp-period} and using \eqref{integral of mu to get beta function}, we get
\begin{gather*}
\frac{1}{2\pi} \int_{\alpha_k} \phi_i \frac{\omega}{df} =
\begin{cases}
\dfrac{1}{h} \eta^{m_i k}\dfrac{\lambda^{m_i/h}}{m_i/h} , &
\mbox{if } 1\leq i\leq N-1, \\
2\mathbf{i} \lambda^{1/2}, &
\mbox{if } i=N.
\end{cases}
\end{gather*}
Therefore,
\begin{gather*}
I^{(-1)}_{\alpha_k}(\lambda) = 2\sum_{i=1}^{N-1}
\eta^{m_i k}\frac{\lambda^{m_i/h}}{m_i/h}\, \phi_{N-i}
-2\mathbf{i} \lambda^{1/2} \phi_N.
\end{gather*}
Note that $\theta(\phi_i) =\big(\tfrac{m_{N-i}}{h} -\tfrac{1}{2}\big)
\phi_i$ for $1\leq i\leq N-1$ and $\theta(\phi_N) = 0$. Therefore
\begin{gather}\label{D-psi}
\Psi(\alpha_k) = 2\sum_{i=1}^{N-1}
\eta^{m_i k}\Gamma(m_i/h) \phi_{N-i} -
\mathbf{i} \Gamma(m_N/h) \phi_N,
\end{gather}
where $m_N=N-1$.
\begin{Remark}
The numbers $\tfrac{m_i}{h}=\tfrac{2i-1}{h}$ ($1\leq i\leq N-1$),
$\tfrac{m_N}{h}=\tfrac{1}{2}$ are the Coxeter exponents.
\end{Remark}
Put
\begin{gather*}
v_k=2\sum_{i=1}^{N-1}
\eta^{m_i k}\Gamma(m_i/h) \phi_{N-i},\qquad 1\leq k\leq N-1,
\end{gather*}
and $v_N=\mathbf{i} \Gamma(m_N/h) \phi_N$.
\begin{Proposition}
The image of the Milnor lattice under the map $\Psi$ is the lattice
in $H_f$ with $\mathbb{Z}$-basis
\begin{gather*}
\beta_1=v_1-v_2,\dots, \beta_{N-1} = v_{N-1} -v_N, \qquad \beta_N=v_{N-1}+v_N.
\end{gather*}
\end{Proposition}
\begin{proof}
Using formula \eqref{ip-formula}, it is straightforward to check
that $\{v_i\}_{1\leq i\leq N}$ is an orthonormal basis of $H_f$
with respect to the intersection
pairing, that is, $(v_i|v_j)=\delta_{i,j}$. We have
$\Psi(\alpha_k)=v_k-v_N$ and $\Psi(\sigma
\alpha_k)-\Psi(\alpha_{k+1}) =2v_N$. Therefore, $\beta_i$ belongs to
the image of the Milnor lattice. On the other hand, since
$(\beta_i|\beta_i)=2$, we get that $\beta_i$ is the image of a vanishing
cycle. Recalling the root system interpretation of the set of
vanishing cycles, we get that $\beta_i$ ($1\leq i\leq N$) are simple
roots and that the
corresponding Dynkin diagram is the Dynkin diagram of type
$D_N$. Since the Milnor lattice is spanned by the set of vanishing
cycles, the claim of the proposition follows.
\end{proof}

\subsection[$E_6$-singularity]{$\boldsymbol{E_6}$-singularity}

Let us fix the following basis of $H_f$:
\begin{gather*} \phi_i=\begin{cases}x_2^{i-1}&\text{if }1\leqslant i\leqslant
 3,\\x_1x_2^{i-4}&\text{if }4\leqslant i\leqslant 6.\end{cases}\end{gather*}
The residue pairing takes the form
\begin{gather*} (\phi_i,\phi_j)=\frac1{2h}\delta_{i+j,7},\qquad 1\leqslant i,j\leqslant 6,\end{gather*}
where $h=12$ is the Coxeter number.
The Riemann surface $M_\mu$ for $\mu\ne 0$ is a non-singular curve in
$\mathbb C^2$ defined by the equation $x_1^3+x_2^4=\mu$.
The projection $(x_1,x_2)\mapsto x_2$ defines a~degree~$3$ branched
covering $M_\mu\to \mathbb C$, with branching points
$x_{2,k}=\mu^{\frac14}\mathbf{i}^k$, $k\in\mathbb Z_4$.

Let $L_k$ ($k\in\mathbb Z_4$) be the line segment
$\big[0,\mu^{\frac14}\mathbf{i}^k\big]$.
Let $A^\prime_k$ be a loop in the $x_2$-plane $\mathbb C$ going around
the branch points $x_{2,k}$ and $x_{2,k+1}$ in the following way: the
loop starts at $0$, it goes along the line segment $L_k$, just before
hitting the branch point $x_{2,k}$ it makes a small loop $C_k$
counterclockwise around $x_{2,k}$, it returns back to the starting
point along $L_k$; then the loop travels in a similar fashion along
$L_{k+1}$ except that this time we make a small loop $C_{k+1}^{-1}$ in
{\em clockwise} direction around $x_{2,k+1}$.
Clearly, the loop $A_{k}^\prime= L_{k+1}^{-1}\circ C_{k+1}^{-1}\circ
L_{k+1}\circ L_k^{-1}\circ C_k\circ L_k$ lifts to three loops
$A_{k,a}$, $a\in\mathbb Z_3$ in $M_{\mu}$, depending on how we choose
the lift of the base point, i.e., the $x_1$-coordinate of the lift of
the base point of $A^\prime_k$ could take the following values:
$x_{1,a}=\mu^\frac13\eta_3^{a}$, where $\eta_3:=e^{\frac{2\pi\mathbf{i}}{3}}$.

Let us consider the loops $A_{k,a}$ with $a\in\mathbb
Z_3\setminus\{0\}$ and $k\in\mathbb Z_4\setminus\{0\}.$
Let us compute the periods of the holomorphic forms
\begin{gather*}
 \phi_i(x_1,x_2)\frac{dx_1dx_2}{dg}=\begin{cases}\dfrac{x_2^{i-1}dx_2}{3x_1^2}
 & \text{if }1\leqslant i\leqslant 3, \vspace{1mm}\\ \dfrac{x_2^{i-4}dx_2}{3x_1}
 & \text{if }4\leqslant i\leqslant 6, \end{cases}
\end{gather*}
along the cycles $A_{k,a}$. As a byproduct of our
computation we will get that the homology classes of these 6 loops
form a basis of $H_1(M_\mu;\mathbb{Z})$. Let us parametrize $A_k^\prime$ as follows:
\begin{gather*}
L_k\colon \ x_2=\mathbf{i}^k\mu^{\frac14}t^\frac 14, \qquad
0\leqslant t\leqslant\left(1-\frac{\epsilon}{\mu^\frac14}\right)^4,\\
C_k\colon \ x_2=\mathbf{i}^k\mu^{\frac14}+\epsilon e^{\mathbf{i}\theta}, \qquad
\frac{k-2}2\pi\leqslant \theta\leqslant \frac{k+2}2\pi.
\end{gather*}
The integrals along the lifts of $C_k$ contribute to the period
integral terms of orders
\begin{gather*}
\begin{cases}
 O\big(\epsilon^{\frac13}\big) &\text{if }1\leqslant i\leqslant 3, \\
 O\big(\epsilon^{\frac23}\big) &\text{if }4\leqslant i\leqslant 6.
\end{cases}
\end{gather*}
These terms vanish in the limit $\epsilon\to 0$. Therefore, under this limit, the periods of the holomorphic forms
\begin{gather*}
\int_{A_{k,a}}\phi_i(x_1,x_2)\frac{dx_1dx_2}{dg}=\begin{cases}\displaystyle \big(1-\eta_3^{-2}\big)\left(\int_{L_{k,a}}-\int_{L_{k+1,a}}\right) \frac{\phi_idx_2}{3x_1^2}&\text{if }1\leqslant i\leqslant3,\vspace{1mm}\\ \displaystyle \big(1-\eta_3^{-1}\big)\left(\int_{L_{k,a}}-\int_{L_{k+1,a}}\right)\frac{\phi_idx_2}{3x_1^2}&\text{if }4\leqslant i\leqslant6,\end{cases}
\end{gather*}
where $\eta_3:=e^{\frac{2\pi\mathbf{i}}{3}}$ and the integral
\begin{gather*}
\int_{L_{k,a}}\frac{\phi_idx_2}{3x_1^2}=
\begin{cases}\displaystyle \int_0^1\frac{\mathbf{i}^{ki}\mu^\frac
 i4t^{\frac{i}4-1}dt}{12\mu^{\frac23}(1-t)^\frac23\eta_3^{2a}}=\frac{\mathbf{i}^{ki}}{12}\eta_3^a\mu^{\frac
 i4-\frac23}B\left(\frac i4,\frac13\right)&\text{if }1\leqslant i\leqslant3,\vspace{1mm}\\
\displaystyle \int_0^1\frac{\mathbf{i}^{k(i-3)}\mu^\frac
 {i-3}4t^{\frac{i-3}4-1}dt}{12\mu^{\frac13}(1-t)^\frac13\eta_3^{a}}=\frac{\mathbf{i}^{k(i-3)}}{12}\eta_3^{2a}\mu^{\frac
 {i-3}4-\frac13}B\left(\frac {i-3}4,\frac23\right)&\text{if }4\leqslant
i\leqslant6.
\end{cases}
\end{gather*}
Then,
\begin{gather*}
\int_{A_{k,a}}\phi_i(x_1,x_2)\frac{dx_1dx_2}{dg}\\
\qquad{} =
\begin{cases}\displaystyle \big(1-\eta_3^{-2}\big)\big(1-\mathbf{i}^i\big)\frac{\mathbf{i}^{ki}}{12}\eta_3^a\mu^{\frac
 i4-\frac23}B\left(\frac i4,\frac13\right)&\text{if }1\leqslant
 i\leqslant3,\\
\displaystyle \big(1-\eta_3^{-1}\big)\big(1-\mathbf{i}^{i-3}\big)\frac{\mathbf{i}^{k(i-3)}}{12}\eta_3^{2a}\mu^{\frac
 {i-3}4-\frac13}B\left(\frac {i-3}4,\frac23\right)&\text{if }4\leqslant
 i\leqslant6.
\end{cases}
\end{gather*}
Let $\alpha_{k,a}=\Sigma A_{k,a}$ be the suspension. Recalling formula
\eqref{susp-period} and using \eqref{integral of mu to get beta function}
\begin{gather*}\big(I^{(-1)}_{\alpha_{k,a}}(\lambda),\phi_{i}\big):= \frac1{2\pi}\int_{\alpha_{k,a}}\phi_i\frac{\omega}{df}=\frac1\pi\partial_\lambda \int_0^\lambda(\lambda-\mu)^\frac12\int_{A_{k,a}}\phi_i(x_1,x_2)\frac{dx_1dx_2}{dg}d\mu\\
\hphantom{\big(I^{(-1)}_{\alpha_{k,a}}(\lambda),\phi_{i}\big)}{}
= \begin{cases}
\displaystyle\frac{\sqrt 3}{12\pi}\eta_3^ae^{-\frac\pi 6\mathbf{i}}\lambda^{\frac i4-\frac16}\frac{\Gamma\big(\frac32\big)\Gamma\big(\frac i4\big)\Gamma\big(\frac13\big)}{\Gamma\big(\frac i4+\frac56\big)}\mathbf{i}^{ki}\big(1-\mathbf{i}^i\big)&\text{if }1\leqslant i\leqslant3,\vspace{1mm}\\
\displaystyle \frac{\sqrt 3}{12\pi}\eta_3^{2a}e^{\frac\pi 6\mathbf{i}}\lambda^{\frac {i-3}4+\frac16}\frac{\Gamma(\frac32)\Gamma\big(\frac {i-3}4\big)\Gamma\big(\frac23\big)}{\Gamma\big(\frac {i-3}4+\frac76\big)}\mathbf{i}^{k(i-3)}\big(1-\mathbf{i}^{i-3}\big)&\text{if }4\leqslant i\leqslant6.
\end{cases}
\end{gather*}
Recalling the formulas for the residue pairing in the basis
$\{\phi_i\}$ we get
\begin{gather*}
I^{(-1)}_{\alpha_{k,a}}(\lambda)=2h\sum_{i=1}^6\big(I^{(-1)}_{\alpha_{k,a}}(\lambda),\phi_{7-i}\big)\phi_i. \end{gather*}
Recalling formula~\eqref{psi-def} and using that by definition
\[
\theta(\phi_i)=\begin{cases}\displaystyle \left(\frac 23-\frac i4\right)\phi_i&\text{if
 }1\leqslant i\leqslant3,\\ \displaystyle \left(\frac13-\frac{i-3}4\right)\phi_i&\text{if
 }4\leqslant i\leqslant6,\end{cases}
 \] we get
\begin{gather}
\Psi(\alpha_{k,a})= \sqrt{\frac
 3\pi}\sum_{i=1}^3e^{\frac\pi
 6\mathbf{i}}{\eta_3^{2a}}\Gamma\left(1-\frac
 {i}4\right)\Gamma\left(\frac23\right)\mathbf{i}^{-ki}\big(1-\mathbf{i}^{-i}\big)\phi_i\notag\\
\hphantom{\Psi(\alpha_{k,a})=}{} + \sqrt{\frac 3\pi}\sum_{i=4}^6e^{-\frac\pi
 6\mathbf{i}}{\eta_3^{a}}\Gamma\left(1-\frac
 {i-3}4\right)\Gamma\left(\frac13\right)\mathbf{i}^{k(3-i)}\big(1-\mathbf{i}^{3-i}\big)\phi_i .\label{E6-psi}
\end{gather}
Let us also point out that by using formula \eqref{Psi of classical
 monodromy operator}, we get the following formulas for the classical
monodromy operator:
\[ \sigma(\alpha_{k,a})=-\alpha_{k+1,a+1},\qquad k\in\mathbb Z_4, \qquad a\in\mathbb Z_3. \]
Recalling formula \eqref{ip-formula},
we get that the intersection pairing
\begin{align*}(\alpha_{k,a}|\alpha_{l,b})& =\frac1\pi(\Psi(\alpha_{k,a}),\cos(\pi\theta)\Psi(\alpha_{l,b}))\\
& =\frac18\sum_{i=1}^3\big(\eta_3^{b-a}\mathbf{i}^{(l-k)i}+\eta_3^{a-b}\mathbf{i}^{(k-l)i}\big)\frac{\cos\big(\big(\frac
 i4-\frac23\big)\pi\big)}{\sin\big(\frac i 4\pi\big)\sin\big(\frac\pi
 3\big)}\big(2-\mathbf{i}^{-i}-\mathbf{i}^{i}\big)\\
& =\sum_{i=1}^3\cos\left(\frac23(b-a)\pi+\frac i2(l-k)\pi\right)\frac{\cos\big(\big(\frac
 i4-\frac23\big)\pi\big)}{\sin\big(\frac\pi 3)}\sin\left(\frac i 4\pi\right).
\end{align*}
Let us identify $\mathbb{Z}_3\setminus\{0\}=\{1,2\}$ and
$\mathbb{Z}_4\setminus\{0\}=\{1,2,3\}$. Every $1\leq a'\leq 6$ can be written
uniquely in the form $a'=3(a-1)+k$, where $1\leq a\leq 2$ and $1\leq
k\leq 3$. Let us define $\alpha_{a'}:=\alpha_{k,a}$. The
intersection pairings $(\alpha_{a^\prime}|\alpha_{b^\prime})$ are
straightforward to compute using the formula from above. We get that
$(\alpha_{a^\prime}|\alpha_{b^\prime})$ coincides with the
$(a',b')$-entry of the following matrix:
\[\begin{pmatrix}
2 & -1 & 0 & -1 & 0 & 0\\
-1 & 2 & -1 & 1 & -1 & 0\\
0 & -1 & 2 & 0 & 1 & -1\\
-1 & 1 & 0 & 2 & -1 & 0\\
0 & -1 & 1 & -1 & 2 & -1\\
0 & 0 & -1 & 0 & -1 & 2
\end{pmatrix}.
\]
The above matrix has determinant $3$. Therefore, the set $\{\alpha_{k,a}\, |\,
1\leq a\leq 2,\, 1\leq k\leq 3\}$ is a set of linearly independent
vanishing cycles. Since the set of all vanishing cycles is a root
system of type $E_6$ and the determinant of the Cartan matrix of a
root system of $E_6$ is also~3, we get that the $\{\alpha_{k,a}\}$ is a set
of simple roots. In particular, it is a $\mathbb{Z}$-basis of the Milnor lattice.

\subsection[$E_7$-singularity]{$\boldsymbol{E_7}$-singularity}

Let us fix the following basis of $H_f$:
\[ \phi_i=\begin{cases}x_1^{i-1}&\text{if }1\leqslant i\leqslant 3,\\x_2x_1^{i-4}&\text{if }4\leqslant i\leqslant 6,\\x_2^2&\text{if }i=7. \end{cases}\]
The residue pairing takes the form
\[ (\phi_i,\phi_j)=\begin{cases}\frac1h\delta_{i+j,7},& 1\leqslant i,j\leqslant 6,\\
-\frac1{6},& i=j=7,\\
0, & \text{otherwise},\end{cases}\]
where $h=18$ is the Coxeter number.
The Riemann surface $M_\mu$ for $\mu\ne 0$ is a non-singular curve in
$\mathbb C^2$ defined by the equation $x_1^3+x_1x_2^3=\mu$. The projection
$(x_1,x_2)\mapsto x_1$ defines a degree $3$ branched covering
$M_\mu\to \mathbb C^*$, with branching points
$x_{1,k}=\mu^{\frac13}\eta_3^k$ $(0\leqslant k\leqslant 2)$, where
$\eta_3=e^{\frac23\pi\mathbf{i}}$.

The method of constructing loops in $M_\mu$ is similar to that of $D_N$-singularity.
Let $A'_k$ be a~simple loop in $\mathbb{C}^*$ around the line segment
$L_k:=[0, x_{1,k}]$, that is, $A_k'$ is a loop starting at a~point~$\tilde\epsilon\eta_3^k$
($0<\tilde\epsilon \ll 1$) on the line segment~$L_k$
sufficiently close to $0$, it goes along the line segment $L_k$, just
before hitting the branch point $x_{1,k}$ it makes a small loop $C_k$
counterclockwise around it, it returns back to the starting point
along the line segment~$L_k$, and finally it makes a small loop $C_0$
counterclockwise around~$0$. Clearly, the loop $A_k'=C_0\circ L_k^{-1}
\circ C_k\circ L_k$ lifts to a~loop~$A_{k,a}$, $a=0,1,2$ in
$M_\mu$, where $a$ indicates the lift of the base point, that is, the
base point is lifted to
$(x_{1,k,a}=\widetilde{\epsilon}\eta_3^k,
x_{2,k,a}=\big(\frac{\mu-\widetilde{\epsilon}^3}{\widetilde{\epsilon}}\big)^\frac13\eta_3^{a-\frac
 k3})$.
We will compute the period integrals along~$A_{k,a}$ for $0\leq k, a
\leq 2$. As a biproduct of our computation we will get that the
following set of 7 loops
$\{A_{0,1},A_{1,1},A_{2,1},A_{0,2},A_{1,2},A_{2,2},A_{0,0}\}$
represents a basis of $H_1(M_\mu;\mathbb{Z})$.

Let us compute the periods of the holomorphic forms
\begin{gather*}
\phi_i(x_1,x_2) \frac{dx_1 dx_2}{dg} =
\begin{cases}
-\dfrac{x_1^{i-2} dx_1}{3x_2^2} & \mbox{if } 1\leq i\leq 3 , \vspace{1mm}\\
-\dfrac{x_1^{i-5}}{3x_2}dx_1 & \mbox{if } 4\leq i\leq 6, \vspace{1mm}\\
-\dfrac{dx_1}{3x_1} & \mbox{if } i=7,
\end{cases}
\end{gather*}
along the cycle $A_{k,a}$. Let us parametrize $A_k^\prime$ as follows:
\begin{gather*}
C_0\colon \ x_1=\tilde\epsilon e^{\mathbf{i} \theta},\qquad
\frac{2k}3\pi\leq \theta \leq \frac{2k+6}3\pi, \\
L_k\colon \ x_1=\eta_3^k\mu^{\frac13}t^\frac 13, \qquad
\frac{\epsilon^3}\mu\leqslant t\leqslant\left(1-\frac{\epsilon}{\mu^\frac13}\right)^3,\\
C_k\colon \ x_1=\eta_3^k\mu^{\frac13}+\epsilon e^{\mathbf{i}\theta}, \qquad
\frac{2k-3}3\pi\leqslant \theta\leqslant \frac{2k+3}3\pi.
\end{gather*}
The integrals along the lifts of $C_0$ and $C_k$ contribute to the
period integral terms of orders
respectively
\[
\begin{cases}O\big(\tilde\epsilon^{i-1+\frac23}\big) \quad \text{and} \quad O\big(\epsilon^{1-\frac23}\big)& \text{if }1\leq i\leq3,
 \\O\big(\tilde\epsilon^{i-4+\frac13}\big) \quad \text{and} \quad O\big(\epsilon^{1-\frac13}\big)& \text{if }4\leq i\leq6, \\ O\big(\tilde
 \epsilon^{0}\big) \quad \text{and} \quad O\big(\epsilon^{1}\big)& \text{if } i=7.
\end{cases}
\]
In the limit $\tilde \epsilon, \epsilon \to 0$ all integrals along the
loops $C_0$ and $C_k$ vanish except for the integral along~$C_0$ when
$i=7$. The latter however is straightforward to compute. Therefore,
after passing to the limit $\epsilon,\tilde{\epsilon}\to 0$, we
get
\[ \int_{A_{k,a}}\phi_i(x_1,x_2)\frac{dx_1dx_2}{dg}=\begin{cases}
\displaystyle \big(1-\eta_3^{-2}\big)\int_{L_{k,a}}\frac{-\phi_idx_1}{3x_1x_2^2} & \mbox{if } 1\leq i\leq 3 , \\
\displaystyle \big(1-\eta_3^{-1}\big)\int_{L_{k,a}}\frac{-\phi_idx_1}{3x_1x_2^2} & \mbox{if } 4\leq i\leq 6, \\
\displaystyle -\frac{2\pi\mathbf{i}}{3} & \mbox{if } i=7, \end{cases}
\]
where $\eta_3:=e^{\frac{2\pi\mathbf{i}}{3}}$ and the integral
\begin{gather*}
\int_{L_{k,a}}\frac{-\phi_idx_1}{3x_1x_2^2}=
\begin{cases}
\displaystyle \int_0^1\frac{-\eta_3^{k(i-\frac13)-2a}\mu^{\frac
 13(i-\frac73)}t^{\frac13(i-\frac{10}3)}dt}{9(1-t)^\frac23}\vspace{1mm}\\
\displaystyle \qquad{} =-\frac{\eta_3^{k(i-\frac13)-2a}\mu^{\frac
 13(i-\frac73)}}{9}B\left(\frac i3-\frac19,\frac13\right)&\text{if
 }1\leqslant i\leqslant3,\vspace{1mm}\\
\displaystyle \int_0^1\frac{-\eta_3^{k(i-4+\frac13)-a}\mu^{\frac
 13(i-4-\frac23)}t^{\frac13(i-4-\frac{8}3)}dt}{9(1-t)^\frac13}\vspace{1mm}\\
\displaystyle \qquad{} =-\frac{\eta_3^{k(i-4+\frac13)-a}\mu^{\frac
 13(i-4-\frac23)}}{9}B\left(\frac {i-4}3+\frac19,\frac23\right)&\text{if
}4\leqslant i\leqslant6. \end{cases}
\end{gather*}
Let $\alpha_{k,a}=\Sigma A_{k,a}$ be the suspension. Recalling formula
\eqref{susp-period} and using \eqref{integral of mu to get beta function}
\begin{align*}
\big(I^{(-1)}_{\alpha_{k,a}}(\lambda),\phi_{i}\big):={} &\frac1{2\pi}\int_{\alpha_{k,a}}\phi_i\frac{\omega}{df}=\frac1\pi\partial_\lambda\int_0^\lambda(\lambda-\mu)^\frac12\int_{A_{k,a}}\phi_i(x_1,x_2)\frac{dx_1dx_2}{dg}d\mu\\
={} &\begin{cases}\displaystyle-\frac{\sqrt 3}{9\pi}\eta_3^{k(i-\frac13)-2a}e^{-\frac\pi 6\mathbf{i}}\lambda^{\frac i3-\frac79+\frac12}\frac{\Gamma(\frac32)\Gamma(\frac i3-\frac19)\Gamma(\frac13)}{\Gamma(\frac i3-\frac79+\frac32)}&\text{if }1\leqslant i\leqslant3,\vspace{1mm}\\
\displaystyle-\frac{\sqrt 3}{9\pi}\eta_3^{k(i-4+\frac13)-a}e^{\frac\pi 6\mathbf{i}}\lambda^{\frac {i-4}3-\frac29+\frac12}\frac{\Gamma(\frac32)\Gamma(\frac {i-4}3+\frac19)\Gamma(\frac23)}{\Gamma(\frac {i-4}3+\frac79+\frac12)}&\text{if }4\leqslant i\leqslant6, \vspace{1mm}\\
\displaystyle -\frac23\mathbf{i}\lambda^\frac12&\text{if }i=7.
\end{cases}
\end{align*}
Recalling the formulas for the residue pairing in the basis
$\{\phi_i\}$ we get
\[
I^{(-1)}_{\alpha_{k,a}}(\lambda)=h\sum_{i=1}^6\big(I^{(-1)}_{\alpha_{k,a}}(\lambda),\phi_{7-i}\big)\phi_i+4\mathbf{i}\lambda^\frac12\phi_7.
\]
Recalling formula \eqref{psi-def} and using that
\[
\theta(\phi_i)=
\begin{cases}
\displaystyle \left(\frac 49-\frac {i-1}3\right)\phi_i&\text{if }1\leqslant i\leqslant3,\vspace{1mm}\\
\displaystyle \left(\frac29-\frac{i-4}3\right)\phi_i&\text{if }4\leqslant i\leqslant6,\\
0&\text{if }i=7,\end{cases}
\]
we get
\begin{gather}
\Psi(\alpha_{k,a})= -\sqrt{\frac 3\pi}
\sum_{i=1}^3e^{\frac\pi 6\mathbf{i}}{\eta_3^{k(\frac13-i)-a}}
\Gamma\left(\frac {3-i}3+\frac19\right)\Gamma\left(\frac23\right)\phi_i\notag\\
\hphantom{\Psi(\alpha_{k,a})=}{} -\sqrt{\frac 3\pi}
\sum_{i=4}^6e^{-\frac\pi 6\mathbf{i}}{\eta_3^{k(4-i-\frac13)-2a}}
\Gamma\left(\frac{7-i}3-\frac19\right)\Gamma\left(\frac13\right)\phi_i
+2\sqrt{-\pi}\phi_7.\label{E7-psi}
\end{gather}
Recalling formula \eqref{Psi of classical monodromy operator}, we get
the following formulas for the classical monodromy operator:
\[
\sigma(\alpha_{k,a})=-\alpha_{k+1,a+1},\qquad k, a\in\mathbb Z_3.
\]
Using formula \eqref{ip-formula}, we get that the intersection pairing
\begin{align*}(\alpha_{k,a}|\alpha_{l,b})& =\frac1\pi(\Psi(\alpha_{k,a}),\cos(\pi\theta)\Psi(\alpha_{l,b}))\\
& =\frac23+\frac16\sum_{i=1}^3\Big(\eta_3^{b-a+(l-k)(i-\frac13)}+\eta_3^{a-b+(k-l)(i-\frac13)}\Big)\frac{\cos\big(\big(\frac i3-\frac79\big)\pi\big)}{\sin\big(\big(\frac i 3-\frac19\big)\pi\big)\sin\big(\frac\pi 3\big)}\\
& =\frac23+\frac13\sum_{i=1}^3\cos\left(\frac{2\pi}3\left({a-b+(k-l)\left(i-\frac13\right)}\right)\right)\frac{\cos\big(\big(\frac i3-\frac79\big)\pi\big)}{\sin\big(\big(\frac i 3-\frac19\big)\pi\big)\sin\big(\frac\pi 3\big)}.
\end{align*}
Let us identify $\mathbb{Z}_3=\{0,1,2\}$. Every $1\leq a'\leq 7$ can be
written uniquely in the form $a^\prime=3(a-1)+k+1$, $1\leq a\leq 3$,
$0\leq k\leq 2$. Put $\alpha_{a'}:=\alpha_{k,\overline{a}}$, where
$0\leq \overline{a}\leq 2$ is the remainder of $a$ modulo 3. Using
the above formula, we get that the intersection pairing in the basis
$\{\alpha_{a'}\}_{1\leq a'\leq 7}$ takes the form
\[\begin{pmatrix}
2&1&1&0&1&1&0\\
1&2&1&0&0&1&1\\
1&1&2&0&0&0&1\\
0&0&0&2&1&1&0\\
1&0&0&1&2&1&0\\
1&1&0&1&1&2&0\\
0&1&1&0&0&0&2
\end{pmatrix}.
\]
The above matrix has determinant $2$. Since the determinant of the
Cartan matrix of the root system of type $E_7$ is also 2, the conclusion is the
same as in the case of $E_6$-singularity, that is, the cycles
$(\alpha_1,\dots,\alpha_7)=(\alpha_{0,1},
\alpha_{1,1},\alpha_{2,1},\alpha_{0,2},\alpha_{1,2},
\alpha_{2,2},\alpha_{0,0})$ form a $\mathbb{Z}$-basis of the Milnor lattice
and hence their images under $\Psi$, computed by formula
\eqref{E7-psi}, give a $\mathbb{Z}$-basis for the image of the Milnor lattice
in $H_f$.

\subsection[$E_8$-singularity]{$\boldsymbol{E_8}$-singularity}\label{sec:e8}

Let us fix the following basis of $H_f$:
\[\phi_i=\begin{cases}x_2^{i-1}&\text{if }1\leqslant i\leqslant 4,\\x_1x_2^{i-5}&\text{if }5\leqslant i\leqslant 8.\end{cases}\]
The residue pairing takes the form
\[(\phi_i,\phi_j)=\frac1{h}\delta_{i+j,9},\qquad 1\leqslant i,j\leqslant 8,\]
where $h=30$ is the Coxeter number.
The Riemann surface $M_\mu$ for $\mu\ne 0$ is a non-singular curve in $\mathbb C^2$ defined by equation $x_1^3+x_2^5=\mu$. The projection $(x_1,x_2)\mapsto x_2$ defines a degree $3$ branched covering $M_\mu\to \mathbb C$, with branching points $x_{2,k}=\mu^{\frac15}\eta_5^k$, $k\in\mathbb Z_5$, where $\eta_5=e^{\frac25\pi\mathbf{i}}$.

The method for constructing loops in $M_\mu$ is almost the same as
that for $E_6$-singularity. Let us omit the similar narration, i.e.,
we define the loops $A_{k,a}$ in the same way, except that now
$a\in\mathbb Z_3\setminus\{0\}$ and $k\in\mathbb
Z_5\setminus\{0\}$. We will see that the homology classes of these 8
loops form a~$\mathbb{Z}$-basis of $H_1(M_\mu;\mathbb{Z})$.
Let us compute the periods of the holomorphic forms
\[ \phi_i(x_1,x_2)\frac{dx_1dx_2}{dg}=\begin{cases}\dfrac{x_2^{i-1}dx_2}{3x_1^2} & \text{if }1\leqslant i\leqslant 4, \vspace{1mm}\\ \dfrac{x_2^{i-5}dx_2}{3x_1} & \text{if }5\leqslant i\leqslant 8, \end{cases}\]
along the cycle $A_{k,a}$. Let us parametrize $A_k^\prime$ as follows:
\begin{gather*}
L_k\colon \ x_2=\eta_5^k\mu^{\frac15}t^\frac 15, \qquad 0\leqslant
 t\leqslant\bigg(1-\frac{\epsilon}{\mu^\frac15}\bigg)^5,\\
C_k\colon \ x_2=\eta_5^k\mu^{\frac15}+\epsilon e^{\mathbf{i}\theta},
 \qquad \frac{2k-5}5\pi\leqslant \theta\leqslant
 \frac{2k+5}5\pi.
\end{gather*}
The integrals along the lifts of $C_k$ contribute to the period
integral terms of orders
$O\big(\epsilon^{\frac13}\big) $ if $1\leqslant i\leqslant 3$ and
$O\big(\epsilon^{\frac23}\big)$ if $4\leqslant i\leqslant 6$.
These terms vanish in the limit $\epsilon\to 0$. Therefore, under this limit, the periods of the holomorphic forms
\[ \int_{A_{k,a}}\phi_i(x_1,x_2)\frac{dx_1dx_2}{dg}=\begin{cases}
\displaystyle \big(1-\eta_3^{-2}\big)\left(\int_{L_{k,a}}-\int_{L_{k+1,a}}\right)\frac{\phi_idx_2}{3x_1^2}&\text{if }1\leqslant i\leqslant4,\vspace{1mm}\\
\displaystyle \big(1-\eta_3^{-1}\big)\left(\int_{L_{k,a}}-\int_{L_{k+1,a}}\right)\frac{\phi_idx_2}{3x_1^2}&\text{if }5\leqslant i\leqslant8,\end{cases}
\]
where $\eta_3:=e^{\frac{2\pi\mathbf{i}}{3}}$ and the integral
\begin{gather*}
\int_{L_{k,a}}\frac{\phi_idx_2}{3x_1^2}=\begin{cases}
\displaystyle \int_0^1\frac{\eta_5^{ki}\mu^\frac i5t^{\frac{i}5-1}dt}{15\mu^{\frac23}(1-t)^\frac23\eta_3^{2a}}=\frac{\eta_5^{ki}}{15}\eta_3^a\mu^{\frac i5-\frac23}B\left(\frac i5,\frac13\right)&\text{if }1\leqslant i\leqslant4,\vspace{1mm}\\
\displaystyle \int_0^1\frac{\eta_5^{k(i-4)}\mu^\frac {i-4}5t^{\frac{i-4}5-1}dt}{15\mu^{\frac13}(1-t)^\frac13\eta_3^{a}}=\frac{\eta_5^{k(i-4)}}{15}\eta_3^{2a}\mu^{\frac {i-4}5-\frac13}B\left(\frac {i-4}5,\frac23\right)&\text{if }5\leqslant i\leqslant8. \end{cases}
\end{gather*}
Then,
\begin{gather*}
\int_{A_{k,a}}\phi_i(x_1,x_2)\frac{dx_1dx_2}{dg}\\
\qquad{} =\begin{cases}
\displaystyle \big(1-\eta_3^{-2}\big)\big(1-\eta_5^i\big)
\frac{\eta_5^{ki}}{15}\eta_3^a\mu^{\frac i5-\frac23}B\left(\frac i5,\frac13\right)&\text{if }1\leqslant i\leqslant4,\\
\displaystyle \big(1-\eta_3^{-1}\big)\big(1-\eta_5^{i-4}\big)\frac{\eta_5^{k(i-4)}}{15}\eta_3^{2a}\mu^{\frac {i-4}5-\frac13}B\left(\frac {i-4}5,\frac23\right)&\text{if }5\leqslant i\leqslant8. \end{cases}
\end{gather*}
Let $\alpha_{k,a}=\Sigma A_{k,a}$ be the suspension. Recalling formula
\eqref{susp-period} and using \eqref{integral of mu to get beta function}
\begin{align*}\big(I^{(-1)}_{\alpha_{k,a}}(\lambda),\phi_{i}\big):={} &\frac1{2\pi}\int_{\alpha_{k,a}}\phi_i\frac{\omega}{df}=\frac1\pi\partial_\lambda \int_0^\lambda(\lambda-\mu)^\frac12\int_{A_{k,a}}\phi_i(x_1,x_2)\frac{dx_1dx_2}{dg}d\mu\\
={} &\begin{cases}\displaystyle \frac{\sqrt 3}{15\pi}\eta_3^ae^{-\frac\pi 6\mathbf{i}}\lambda^{\frac i5-\frac16}\frac{\Gamma\big(\frac32\big)\Gamma\big(\frac i5\big)\Gamma\big(\frac13\big)}{\Gamma\big(\frac i5+\frac56\big)}\eta_5^{ki}\big(1-\eta_5^i\big)&\text{if }1\leqslant i\leqslant4,\vspace{1mm}\\
\displaystyle \frac{\sqrt 3}{15\pi}\eta_3^{2a}e^{\frac\pi 6\mathbf{i}}\lambda^{\frac {i-4}5+\frac16}\frac{\Gamma\big(\frac32\big)\Gamma\big(\frac {i-4}5\big)\Gamma\big(\frac23\big)}{\Gamma\big(\frac {i-4}5+\frac76\big)}\eta_5^{k(i-4)}\big(1-\eta_5^{i-4}\big)&\text{if }5\leqslant i\leqslant8.
\end{cases}
\end{align*}
Recalling the formulas for residue pairing, we get
\[ I^{(-1)}_{\alpha_{k,a}}(\lambda)=h\sum_{i=1}^8\big(I^{(-1)}_{\alpha_{k,a}}(\lambda),\phi_{9-i}\big)\phi_i.\]
By definition,
\begin{gather*}
\theta(\phi_i)=
\begin{cases}
\displaystyle \left(\frac 23-\frac i5\right)\phi_i,&
\text{if }1\leqslant i\leqslant4,\vspace{1mm}\\
\displaystyle \left(\frac13-\frac{i-4}5\right)\phi_i,&
\text{if }5\leqslant i\leqslant8.
\end{cases}
\end{gather*}
Therefore, recalling formula \eqref{psi-def}, we get
\begin{gather}
\Psi(\alpha_{k,a})= \sqrt{\frac 3\pi}\sum_{i=1}^4e^{\frac\pi
 6\mathbf{i}}{\eta_3^{2a}}\Gamma\left(1-\frac
 {i}5\right)\Gamma\left(\frac23\right)\eta_5^{-ki}\big(1-\eta_5^{-i}\big)\phi_i\notag\\
\hphantom{\Psi(\alpha_{k,a})=}{} +\sqrt{\frac 3\pi}\sum_{i=5}^8e^{-\frac\pi 6\mathbf{i}}{\eta_3^{a}}\Gamma\left(1-\frac {i-4}5\right)\Gamma\left(\frac13\right)\eta_5^{k(4-i)}\big(1-\eta_5^{4-i}\big)\phi_i.\label{E8-psi}
\end{gather}
Recalling formula \eqref{Psi of classical monodromy operator}, we get
that the following formulas for the classical monodromy operator:
\[ \sigma(\alpha_{k,a})=-\alpha_{k+1,a+1},\qquad k\in\mathbb Z_5, \qquad a\in\mathbb Z_3.
\]
Recalling formula \eqref{ip-formula}, the intersection pairing
\begin{align*}
 (\alpha_{k,a}|\alpha_{l,b})&=\frac1\pi(\Psi(\alpha_{k,a}),\cos(\pi\theta)\Psi(\alpha_{l,b}))\\
& =\frac1{10}\sum_{i=1}^4\big(\eta_3^{b-a}\eta_5^{(l-k)i}+\eta_3^{a-b}\eta_5^{(k-l)i}\big)
 \frac{\cos\big(\big(\frac i5-\frac23\big)\pi\big)}{\sin\big(\frac i 5\pi\big)\sin\big(\frac\pi 3\big)}\big(2-\eta_5^{-i}-\eta_5^{i}\big)\\
& =\frac45\sum_{i=1}^4\cos\left(\frac23(b-a)\pi+\frac {2i}5(l-k)\pi\right)\frac{\cos\big(\big(\frac i5-\frac23\big)\pi\big)}{\sin\big(\frac\pi 3\big)}\sin\left(\frac i 5\pi\right).
\end{align*}
Let us identify $\mathbb{Z}_3\setminus\{0\}=\{1,2\}$ and
$\mathbb{Z}_5\setminus\{0\}=\{1,2,3,4\}$. Every $1\leq a^\prime\leq 8$ can be
written uniquely in the form
$a^\prime=4(a-1)+k$, where $1\leq a\leq 2$ and $1\leq k\leq 4$. Put
$\alpha_{a'}:=\alpha_{k,a}$. Then the intersection matrix
$(\alpha_{a^\prime}|\alpha_{b^\prime})$ takes the following form:
\[\begin{pmatrix}
2 & -1 & 0 & 0 & -1 & 0 & 0 & 0\\
-1 & 2 & -1 & 0 & 1 & -1 & 0 & 0\\
0 & -1 & 2 & -1 & 0 & 1 & -1 & 0\\
0 & 0 & -1 & 2 & 0 & 0 & 1 & -1\\
-1 & 1 & 0 & 0 & 2 & -1 & 0 & 0\\
0 & -1 & 1 & 0 & -1 & 2 & -1 & 0\\
0 & 0 & -1 & 1 & 0 & -1 & 2 & -1\\
0 & 0 & 0 & -1 & 0 & 0 & -1 & 2
\end{pmatrix}.\]
The above matrix has determinant $1$. Since the determinant of the
Cartan matrix of the root system of type $E_8$ is also 1, the
conclusion is the same as in the previous cases.

\section{K-theoretic interpretation}\label{sec:K}

The goal of this section is to prove Theorem~\ref{t1}.
\subsection{Fermat cases}

Let us compute explicitly the map $\operatorname{ch}_\Gamma$ for
$f(x)=f^T(x)= x_1^{a_1}+x_2^{a_2}+x_3^{a_3}$ with $a_3=2$. In fact,
our computation works for arbitrary $a_3$ as well, except for one
small technical detail, that is, we will prove that the group $K^{-1}_{G^T}\big(V^T_1\big)$ is
torsion free. This fact should be true for any positive integer $a_3$,
but the argument that we give works only if $a_3=2$. The group
\begin{gather*}
G^T=\big\{g=(g_1,g_2,g_3)\in (\mathbb{C}^*)\, |\, g_1^{a_1}=g_2^{a_2}=g_3^{a_3}=1\big\}.
\end{gather*}
If $g\in G^T$ is such that $I=\{i\, |\, g_i=1\}$ is a non-empty set,
then it is easy to see that the map $x=(x_1,x_2,x_3)\mapsto
\big(x_i^{a_i}\big)_{i\in I}$ induces isomorphisms
$\operatorname{Fix}_g\big(\mathbb{C}^3\big)/G^T\cong \mathbb{C}^I$ and
$\operatorname{Fix}_g\big(V^T_1\big)/G^T\cong H_I$, where $H_I\subset \mathbb{C}^I$ is
the hyperplane $\sum_{i\in I} y_i=1$. Since the pair $\big(\mathbb{C}^I,H_I\big)$ is
contractible the groups
\begin{gather*}
K^0\big( \operatorname{Fix}_g\big(\mathbb{C}^3\big)/G^T,
\operatorname{Fix}_g\big(V_1^T\big)/G^T\big)=
H^*\big( \operatorname{Fix}_g\big(\mathbb{C}^3\big)/G^T,
\operatorname{Fix}_g\big(V_1^T\big)/G^T\big) = 0.
\end{gather*}
If $g\in G^T$ is such that $g_i\neq 1$ for all $i$, then
$\operatorname{Fix}_g\big(\mathbb{C}^3\big) = \{0\}$ and
$\operatorname{Fix}_g\big(V^T_1\big) = \varnothing$. Note that the number of
such $g$ is $N=(a_1-1)(a_2-1)(a_3-1)$, that is, the multiplicity of
the singularity corresponding to the polynomial~$f$.
\begin{Lemma}\label{le:fermat-tor}
 The group $K^{-1}_{G^T}\big(V_1^T\big)$ is torsion free.
\end{Lemma}
Let us postpone the proof of this lemma until Section~\ref{sec:fermat-tor}.
Note that $K^{-1}_{G^T}\big(V_1^T\big)\otimes \mathbb{C}=0$. Therefore, according to
the above Lemma~\ref{le:fermat-tor}, we have $K^{-1}_{G^T}\big(V^T_1\big)=0$.
The long exact sequence of the pair $\big(\mathbb{C}^3,V^T_1\big)$
yields the following exact sequence
\begin{gather*}
\xymatrix{
0\ar[r] &
K^0_{G^T}\big(\mathbb{C}^3,V^T_1\big)\ar[r] &
K^0_{G^T}\big(\mathbb{C}^3\big) \ar[r] &
K^0_{G^T}\big(V^T_1\big).}
\end{gather*}
On the other hand, $K^0_{G^T}\big(\mathbb{C}^3\big)$ coincides with the
representation ring of $G^T$, that is,
\begin{gather*}
K^0_{G^T}\big(\mathbb{C}^3\big) =
\mathbb{Z}[L_1,L_2,L_3]/\big(L_1^{a_1}-1,L_2^{a_2}-1,L_3^{a_3}-1\big),
\end{gather*}
where $L_i=\mathbb{C}^3\times \mathbb{C}$ is the trivial bundle with $G^T$-action
$g\cdot (x,\lambda):= (gx,g_i\lambda)$. Note that $T\mathbb{C}^3\cong
L_1+L_2+L_3$ in the category of $G^T$-equivariant bundles. We claim that
\begin{gather}\label{fermat-K}
K^0_{G^T}\big(\mathbb{C}^3,V^T_1\big) =
(L_1-1)(L_2-1)(L_3-1)
\mathbb{Z}[L_1,L_2,L_3]/\big(L_1^{a_1}-1,L_2^{a_2}-1,L_3^{a_3}-1\big).
\end{gather}
Indeed, note that $s_i(x)=(x,f_{x_i})$ is a $G^T$-equivariant
section of $L_i^{-1}$. The Koszul complex corresponding to the
sequence $(s_1,s_2,s_3)$ has the form
\begin{gather*}
\xymatrix{
L_1L_2L_3\ar[r] &
\bigoplus_{1\leq i<j\leq 3} L_iL_j \ar[r] &
\bigoplus_{1\leq i\leq 3} L_i \ar[r] &
\underline{\mathbb{C}},}
\end{gather*}
where $\underline{\mathbb{C}}$ is the trivial bundle with trivial
$G^T$-action. The sequence $(s_1,s_2,s_3)$ is regular, so the
corresponding Koszul complex is a resolution of the structure sheaf of
the zero locus $\{s_1=s_2=s_3=0\}$. The zero locus is $\{0\}$ and
since $0\notin V^T_1$ the restriction of the Koszul complex to $V^T_1$
is exact, i.e., the Koszul complex represents an element of
$K^0_{G^T}\big(\mathbb{C}^3,V^T_1\big)$. This proves that the RHS of \eqref{fermat-K}
is a $\mathbb{Z}$-submodule of the LHS. Note that both the LHS and the RHS of~\eqref{fermat-K} are free $\mathbb{Z}$-modules of rank $N$. Therefore, the
quotient of LHS by RHS is a finite Abelian group. In order to prove
that the quotient is 0, it is sufficient
to prove that if $g\in K^0_{G^T}\big(\mathbb{C}^3\big)$ and $m g$ belongs to the RHS
of~\eqref{fermat-K} for some integer~$m$, then $g$ belongs to the RHS
of~\eqref{fermat-K} too. The proof is straightforward so we leave it
as an exercise.

Let us fix the following basis of $K^0_{G^T}\big(\mathbb{C}^3,V^T_1\big)$:
\begin{gather*}
A_{m_1,m_2,m_3}:=
L_1^{m_1}L_2^{m_2}L_3^{m_3}
(L_1-1)(L_2-1)(L_3-1),\qquad 0\leq m_i\leq a_i-2.
\end{gather*}
Let $e_{k_1,k_2,k_3}=1\in H^0\big(\operatorname{Fix}_g\big(\mathbb{C}^3\big)/G,
\operatorname{Fix}_g\big(V^T_1\big)/G \big)$, where
$g=\big(e^{2\pi\mathbf{i} k_1/a_1}, e^{2\pi\mathbf{i} k_2/a_2}, e^{2\pi\mathbf{i} k_3/a_3}\big)$.
We get
\begin{gather*}
\operatorname{ch}_\Gamma(A_{m_1,m_2,m_3}) = \frac{1}{2\pi}
\sum_{k_1=1}^{a_1-1} \sum_{k_2=1}^{a_2-1}\sum_{k_3=1}^{a_3-1}
\prod_{i=1}^3 \left(
\Gamma\left(1-\frac{k_i}{a_i}\right) e^{-2\pi\mathbf{i}
 k_i m_i/a_i }
\big(e^{-2\pi\mathbf{i} k_i/a_i}-1\big) \right) e_{k_1,k_2,k_3},
\end{gather*}
where the ingredients of the above formula are computed as
follows. Since
$\operatorname{Fix}_g\big(\mathbb{C}^3\big)=\{0\}$ and the action of $g$ on
$L_i|_{\operatorname{Fix}_g(\mathbb{C}^3)}$ is given by multiplication by
$e^{2\pi\mathbf{i} k_i/a_i}$ we get
\begin{gather*}
\widehat{\Gamma}(L_i)|_{\operatorname{Fix}_g(\mathbb{C}^3)/G^T}=
\Gamma\left(1-\frac{k_i}{a_i}\right)\qquad\mbox{and}\\
\iota^*\widetilde{\operatorname{ch}}(L_i)|_{\operatorname{Fix}_g(\mathbb{C}^3)/G^T}=
\widetilde{\operatorname{ch}}\big(L^{-1}_i\big)|_{\operatorname{Fix}_g(\mathbb{C}^3)/G^T}=
e^{-2\pi\mathbf{i} k_i/a_i },
\end{gather*}
where we used that $\iota^*(L_i)=L_i^{-1}$. The orbifold tangent
bundle $\big[T\mathbb{C}^3/G^T\big]=L_1+L_2+L_3$ so its $\Gamma$-class is
$\prod_{i=1}^3 \widehat{\Gamma}(L_i)$, while
$\iota^*\widetilde{\operatorname{ch}}|_{\operatorname{Fix}_g(\mathbb{C}^3)/G^T}$
is a ring homomorphism, so the computation of its value on
$A_{m_1,m_2,m_3}$ amounts to the substitution
$L_i\mapsto e^{-2\pi\mathbf{i} k_i/a_i }$.
Let us specialize the above formula to the cases of $A_N$, $E_6$, and
$E_8$ singularities. In the first case $a_1=N+1$, $a_2=a_3=2$. The
above formula takes the form
\begin{gather*}
\operatorname{ch}_\Gamma(A_{m,0,0}) =
2\sum_{k=1}^N \eta^{-km}\big(\eta^{-k}-1\big) \Gamma\left(1-\frac{k}{N+1}\right)
e_{k,1,1}.
\end{gather*}
Comparing with \eqref{A-psi}, we get that if we define
$\operatorname{mir}(\phi_i) = e_{i,1,1}$ ($1\leq i\leq N$), then the
images of~$\Psi$ and $\operatorname{ch}_\Gamma$ will coincide. The
vanishing cycle $\alpha_{k,a}$ corresponds to $(-1)^a A_{k,0,0}$.

For the case of $E_6$ we have $a_1=3$, $a_2=4$, $a_3=2$. The formula
takes the form
\begin{gather*}
\operatorname{ch}_\Gamma(A_{m_1,m_2,0}) =
-\frac{1}{\sqrt{\pi}} \sum_{k_1=1}^2 \sum_{k_2=1}^3
\Gamma\left(1-\frac{k_1}{3}\right)
\Gamma\left(1-\frac{k_2}{4}\right)\\
\hphantom{\operatorname{ch}_\Gamma(A_{m_1,m_2,0}) =}{}\times
\eta_3^{-k_1m_1} \eta_4^{-k_2m_2}
\big(\eta_3^{-k_1}-1\big)\big(\eta_4^{-k_2}-1\big)
e_{k_1,k_2,1},
\end{gather*}
where $\eta_3=e^{2\pi\mathbf{i}/3}$ and $\eta_4=e^{2\pi\mathbf{i}/4}=\mathbf{i}$. Note that
$\eta_3^{-1}-1= -\sqrt{3} e^{\pi\mathbf{i}/6}$ and $\eta_3^{-2}-1= -\sqrt{3}
e^{-\pi\mathbf{i}/6}$. Comparing with \eqref{E6-psi} we get that if we define{\samepage
\begin{gather*}
\operatorname{mir}(\phi_i) =
\begin{cases}
e_{1,i,1}, & \mbox{for } 1\leq i\leq 3, \\
e_{2,i-3,1}, & \mbox{for } 4\leq i\leq 6,
\end{cases}
\end{gather*}
then the images of $\Psi$ and $\operatorname{ch}_\Gamma$ will
coincide. The vanishing cycle $\alpha_{k,a}$ corresponds to
$-A_{a,k,0}$.}

Suppose now that the singularity is of type $E_8$, that is, $a_1=3$,
$a_2=5$, and $a_3=2$. The formula takes the form
\begin{gather*}
\operatorname{ch}_\Gamma(A_{m_1,m_2,0})=
-\frac{1}{\sqrt{\pi}}
\sum_{k_1=1}^2 \sum_{k_2=1}^4
\Gamma\left(1-\frac{k_1}{3}\right)
\Gamma\left(1-\frac{k_2}{5}\right)\\
\hphantom{\operatorname{ch}_\Gamma(A_{m_1,m_2,0})=}{}\times
\eta_3^{-k_1m_1} \eta_5^{-k_2m_2}
\big(\eta_3^{-k_1}-1\big)\big(\eta_5^{-k_2}-1\big) e_{k_1,k_2,1},
\end{gather*}
where $\eta_3=e^{2\pi\mathbf{i}/3}$ and $\eta_5=e^{2\pi\mathbf{i}/5}$. Comparing
with formula \eqref{E8-psi} we get that if we define
\begin{gather*}
\operatorname{mir}(\phi_i) =
\begin{cases}
e_{1,i,1}, & \mbox{for } 1\leq i\leq 4, \\
e_{2,i-4,1}, & \mbox{for } 5\leq i\leq 8,
\end{cases}
\end{gather*}
then the images of $\Psi$ and $\operatorname{ch}_\Gamma$ will
coincide. The vanishing cycle $\alpha_{k,a}$ corresponds to
$-A_{a,k,0}$.

\subsection{Suspension isomorphism}
Suppose that $X$ is a finite CW-complex equipped with an action of a
finite (or more generally compact Lie) group $G$. Let
$\mathbbm{\mu}_2=\{\pm 1\}$ and
\begin{gather*}
\Sigma X=
\bsfrac{\ X\times [-1,1] }{X\times \{-1\}}
\nicefrac{\ }{X\times \{1\}}
\end{gather*}
be the suspension of $X$, where the double quotient simply means that
the quotient is taken in two steps: first by, say, $X\times \{-1\}$ and
then by $X\times \{1\}$. Note that $G\times \mathbbm{\mu}_2$ acts
naturally on~$\Sigma X$ via~$(g,\epsilon)\cdot (x,t)=(gx,\epsilon
t)$ and that the 0-dimensional sphere $\mathbbm{S}^0:=\Sigma
X-{X\times (-1,1)}$ is a~$G\times \mathbbm{\mu}_2$-equivariant
subcomplex of~$\Sigma X$.

Let $L=[-1,1]\times \mathbb{C}$ be the trivial $\mathbbm{\mu}_2$-equivariant line
bundle on the interval $I:=[-1,1]$, where the representation of
$\mathbbm{\mu}_2$ on $\mathbb{C}$ is given by $\epsilon \cdot \lambda =
\epsilon \lambda$. It is an easy and amusing exercise to check
that $K^0_{\mathbbm{\mu}_2}(I,\partial I) = \mathbb{Z}\, \ell$, where $\ell$
is the relative K-theoretic class of the complex
$\xymatrix{L\ar[r]^{t}&\underline{\mathbb{C}}}$, where $\underline{\mathbb{C}}=I\times \mathbb{C}$ is
the trivial $\mathbbm{\mu}_2$-equivariant line bundle
corresponding to the trivial representation of~$\mathbbm{\mu}_2$ on~$\mathbb{C}$ and the map is induced by $(t,\lambda)\mapsto (t,t\lambda)$.
\begin{Lemma}\label{le:susp}
 The exterior tensor product by $\ell$ induces an isomorphism
\[ K^i_G(X)\cong K^i_{G\times \mathbbm{\mu}_2}\big(\Sigma X,\mathbbm{S}^0\big).\]
\end{Lemma}
\begin{proof}
By definition
\begin{gather*}
K^i_{G\times \mathbbm{\mu}_2}\big(\Sigma X,\mathbbm{S}^0\big) =
\widetilde{K}^i_{G\times \mathbbm{\mu}_2}\big(\Sigma X/\mathbbm{S}^0\big) =
K^i_{G\times \mathbbm{\mu}_2} (X\times I, X\times \partial I)
\cong K^i_G(X)\otimes K^0_{\mathbbm{\mu}_2}(I,\partial I),
\end{gather*}
where we used that $K^i_{\mathbbm{\mu}_2}(I,\partial I)$ is isomorphic to~$\mathbb{Z}$
for~$i$ even and~$0$ for $i$ odd, so the last isomorphism is given by
the equivariant K\"unneth formula (see~\cite{MR259897}).
\end{proof}

Suppose now that $Y\subset X$ is a $G$-invariant CW-subcomplex of
$X$. Using the long exact sequence of the triple
$\mathbbm{S}^0\subset \Sigma Y\subset \Sigma X$ and Lemma~\ref{le:susp}, it is straightforward to prove the following corollary.
\begin{Corollary}\label{cor:susp}
The exterior tensor product by $\ell$ induces an isomorphism
\begin{gather*}
K^i_G(X,Y)\cong K^i_{G\times \mathbbm{\mu}_2}(\Sigma X,\Sigma Y).
\end{gather*}
\end{Corollary}

\subsection[The relative K-ring for $D_N$-singularity]{The relative K-ring for $\boldsymbol{D_N}$-singularity}
Let us return to the settings of $D_N$-singularity. We have
$f^T(x)=x_1^2+x_1x_2^{N-1}+x_3^2$.
The group of diagonal symmetries of $f^T$ is
\begin{gather*}
G^T=\big\{t\in (\mathbb{C}^*)^3\, |\, t_1^2=t_1t_2^{N-1} = t_3^2=1\big\}.
\end{gather*}
Let $L_i=\mathbb{C}^3\times \mathbb{C}$ be the $G^T$-equivariant line bundle for
which the action of $G^T$ on $\mathbb{C}$ is given by the character $G^T\to
\mathbb{C}^*$, $(t_1,t_2,t_3)\mapsto t_i$. Let us introduce the following $N$
complexes of $G^T$-equivariant vector bundles on $\mathbb{C}^3$:
\begin{gather*}
E_i^{\bullet}\colon \
\xymatrix{
L_1L_2^{i-1} L_3 \ar[r]^-{d_0} &
L_1L_2^{i-1} \oplus L_2^{i-1} L_3 \ar[r]^-{d_1} &
L_2^{i-1},}\qquad 1\leq i\leq N-1,
\end{gather*}
where the differentials are defined by
$d_0(x,\lambda) = (x,-x_3\lambda, x_1 \lambda)$ and
$d_1(x,\lambda_1,\lambda_3) = (x,x_1 \lambda_1 + x_3\lambda_3)$,
\begin{gather*}
E_N^{\bullet}\colon \
\xymatrix{
L_3\ar[r]^-{d_0}&
\underline{\mathbb{C}}\oplus L_3 \ar[r]^-{d_1} &
\underline{\mathbb{C}},}
\end{gather*}
where
$d_0(x,\lambda) = \big(x,-x_3\lambda, x_1^2 \lambda\big)$ and $d_1(x,\lambda_1,\lambda_3) = \big(x,x_1^2 \lambda_1 + x_3\lambda_3\big)$.
\begin{Proposition}\label{le:rel_K-D}
The relative $K$-ring $K^0_{G^T}\big(\mathbb{C}^3,V^T\big)\cong \mathbb{Z}^N$ and the complexes
$E_i^\bullet$ $(1\leq i\leq N)$ represent a $\mathbb{Z}$-basis.
\end{Proposition}
\begin{proof} Note that the complex $E_i^\bullet$ $(1\leq i\leq N-1)$ is a tensor
product of $L_2^{i-1}$,
$\xymatrix{L_1\ar[r]^-{x_1} & \underline{\mathbb{C}},}$ and
$\xymatrix{L_3\ar[r]^-{x_3} & \underline{\mathbb{C}}}$ and that the complex
$E^\bullet_N$ is a tensor product of
$\xymatrix{\underline{\mathbb{C}} \ar[r]^-{x_1^2} &
\underline{\mathbb{C}}}
$
and
$\xymatrix{L_3\ar[r]^-{x_3} & \underline{\mathbb{C}}.}$
On the other hand, we have
$G^T=A\times \mathbf{\mu}_2$, where
$
A=\big\{t\in (\mathbb{C}^*)^2\, |\, t_1^2=t_1t_2^{N-1} = 1\big\}.
$
Recalling Corollary \ref{cor:susp} we get
$K^0_{G^T}\big(\mathbb{C}^3,V^T\big)\cong K^0_A\big(\mathbb{C}^2,M\big)$,
where $M=\big\{x\in \mathbb{C}^2\, |\, x_1^2+x_1x_2^{N-1}=1\big\}$. Slightly
abusing the notation we denote by $L_1$ and $L_2$ the restriction of
the vector bundles $L_1$ and $L_2$ to $\mathbb{C}^2$. Note that the operation
tensor product by the complex
$\xymatrix{ L_3\ar[r]^-{x_3} & \underline{\mathbb{C}}}$ is precisely the
exterior tensor product by the complex $\ell$ in the suspension
isomorphism from Corollary~\ref{cor:susp}. Therefore, it is sufficient to
prove that the complexes
$\xymatrix{L_1L_2^{i-1}\ar[r]^-{x_1} & L_2^{i-1}}$ ($1\leq i\leq
N-1$) and $\xymatrix{\underline{\mathbb{C}}\ar[r]^-{x_1^2} & \underline{\mathbb{C}}}$
represent a $\mathbb{Z}$-basis of $K^0_A\big(\mathbb{C}^2,M\big)$.

The long exact sequence of the pair $\big(\mathbb{C}^2,M\big)$ yields the following
exact sequence:
\begin{gather*}
\xymatrix{
0\ar[r] &
K^{-1}_A(M)\ar[r]^-{\delta} &
K^0_A\big(\mathbb{C}^2,M\big)\ar[r]^-{\rho} &
K^0_A\big(\mathbb{C}^2\big) \ar[r] &
K^0_A(M)
},
\end{gather*}
where we used that $K^{-1}_A\big(\mathbb{C}^2\big)=0$. We have
\begin{gather*}
K^0_A\big(\mathbb{C}^2\big) = \mathbb{Z}[L_1,L_2]/\big\langle L_1^2-1, L_1 L_2^{N-1}-1\big\rangle,
\end{gather*}
where the RHS is the representation ring of $A$. Just like in the
Fermat cases it is easy to prove that the image of $\rho$ coincides
with the ideal $(L_1-1) K^0_A\big(\mathbb{C}^2\big)$. Note that
$\operatorname{Im}(\rho)\cong \mathbb{Z}^{N-1}$ and that
$\rho(E_i^\bullet)=L_2^{i-1}(L_1-1)$ ($1\leq i\leq N-1$) is a
$\mathbb{Z}$-basis. It remains only to prove that $K^{-1}_A(M)\cong \mathbb{Z}$ and
that $\operatorname{Im}(\delta)$ is generated as a $\mathbb{Z}$-module by the
complex $E_N^\bullet$.

Let us first prove that $K^{-1}_A(M)\cong\mathbb{Z}$. Let $\pi\colon M\to \mathbb{C}^*$ be
the map $(x_1,x_2)\mapsto x_1^2$. The map $\pi$ is a branched covering
with only one branch point, that is, $1\in \mathbb{C}^*$. The corresponding
ramification points are $R=\{(-1,0),(1,0)\}$. Note that $R$ is an
$A$-invariant subset. The idea is to use the
long exact sequence of the pair $(M,M\setminus{R})$. The action of $A$
on $M\setminus{R}$ is free, so we have
\begin{gather*}
K^i_A(M\setminus{R})=K^i((M\setminus{R})/A) = K^i(\mathbb{C}\setminus{\{0,1\}}).
\end{gather*}
Therefore $K^0_A(M\setminus{R})\cong \mathbb{Z}$ and
$K^{-1}_A(M\setminus{R})\cong \mathbb{Z}^2$. The groups
$K^i(M,M\setminus{R})$ are also easy to compute. Let $U\subset M$ be a
small $A$-invariant open neighborhood of $R$. Then by excision
$K_A^i(M,M\setminus{R})=K_A^i(U,U\setminus{R})$. Note that the open
neighborhood $U$ can be identified with an open neighborhood of $R$ in
the normal bundle $\nu_R$ to $R$ in $M$. Indeed, the normal bundle is
trivial $\nu_R=R\times \mathbb{C}$ and a point in $(x_1,x_2)\in U$ satisfies
$x_1=\tfrac{1}{2} \Big({-}x_2^{N-1} \pm \sqrt{x_2^{2N-2}+4}\Big)$, so the map
$U\to \nu_R$, $(x_1,x_2)\mapsto ((\pm 1, 0), x_2)$ identifies $U$ with
an open neighborhood of the zero section $R$ in $\nu_R$. Clearly, the
pullback of $\nu_R$ to $U$ is $L_2$ and the Thom class of $\nu_R$ is
represented as an element of $K^0_A(\nu_R)=K_A^0(U,U\setminus{R})$ by
the complex $\xymatrix{\underline{\mathbb{C}}\ar[r]^-{x_2} & L_2}$.
According to Thom isomorphism $K^i_A(U,U\setminus{R})\cong
K^i_A(R)$. Note that $R$ is an $A$-orbit, that is, $R=A/B$, where $B$
is the cyclic subgroup of $A$ generated by $\big(1,\eta^2\big)$. Therefore,
$K^{-1}_A(R)=0$ and $K^0_A(R)$ coincides with the representation ring
of $B$. Since the Thom isomorphism is given by tensor product with the
Thom class, we get
\begin{gather*}
K^0_A(M,M\setminus{R}) = \bigoplus_{i=1}^{N-1}
\mathbb{Z}\xymatrix{\big[
L_2^{i-1} \ar[r]^-{x_2} & L_2^i \big].}
\end{gather*}
The long exact sequence of the pair $(M,M\setminus{R})$ takes the form
\begin{gather*}
\xymatrix{
0\ar[r] &
K^{-1}_A(M) \ar[r] &
K^{-1}_A(M\setminus{R})\ar[r]^-{\delta} &
K^0_A(M,M\setminus{R}).}
\end{gather*}
We already proved that
$K^{-1}_A(M\setminus{R})\cong K^{-1}\big(\mathbb{C}^2\setminus{\{0,1\} } \big)\cong
\mathbb{Z}^2$. We will make use of the following explicit interpretation of the $K$-group
$K^{-1}_A(\ )$.
By definition, for any finite CW-complex $X$, we have
$K^{-1}_A(X)=\widetilde{K}^0_A(\Sigma(X\sqcup {\rm pt}))$. Since the
complement of $X$ in
$\Sigma(X\sqcup {\rm pt}) $ is contractible, we can
think of an element of $K^{-1}_A(X)$ as a representation of $A$ on
some vector space $\mathbb{C}^r$ and an $A$-equivariant isomorphism
$\phi\colon X\times \mathbb{C}^r\to X\times \mathbb{C}^r$, that is, an $A$-equivariant
morphism $X\to \operatorname{GL}_r(\mathbb{C})$. In our case the elements of
$K^{-1}_A(M\setminus{R})$ are obtained by pullback from
$K^{-1}(\mathbb{C}\setminus{ \{0,1\} })$. The latter is generated by two
elements that correspond, in the way described above, to
the two maps $\mathbb{C}\setminus{\{0,1\}} \to \mathbb{C}^*$, $t\mapsto t$ and
$t\mapsto 1-t$. Therefore, the group $K^{-1}_A(M\setminus{R})$ is
generated by the two elements that correspond to the two maps
$M\setminus{R}\to \mathbb{C}^*$ defined by $(x_1,x_2)\mapsto x_1^2$ and
$(x_1,x_2)\mapsto 1-x_1^2=x_1x_2^{N-1}$. The connecting morphism
$\delta$ can be described as follows. Given an $A$-equivariant isomorphism
$\phi\colon (M\setminus{R})\times \mathbb{C}^r\to (M\setminus{R})\times \mathbb{C}^r$,
then let us pick an extension to a vector bundle
morphism $\widetilde{\phi}\colon M\times \mathbb{C}^r \to M\times \mathbb{C}^r$. The
resulting complex clearly represents an element of
$K^0_A(M,M\setminus{R})$ and that is what $\delta(\phi)$ is. The
extensions in our case are straightforward to construct. We get that
\begin{gather*}
\operatorname{Im}(\delta) = \mathbb{Z}
\xymatrix{\big[\underline{\mathbb{C}}\ar[r]^-{x_1^2} &
\underline{\mathbb{C}} \big]} +
 \mathbb{Z}
\xymatrix{\big[\underline{\mathbb{C}}\ar[r]^-{x_1 x_2^{N-1}} &
\underline{\mathbb{C}}\big]. }
\end{gather*}
Note however, that $x_1\neq 0$ so $x_1^2$ defines an isomorphism,
i.e., the first complex is 0 in \linebreak \mbox{$K^0_A(M,M\setminus{R})$}. In
particular, the kernel of the connecting homomorphism $\delta$ is
$\cong \mathbb{Z}$ and it is generated by the
element in $K^{-1}_A(M\setminus{R})$ corresponding to the map
$M\setminus{R}\to \mathbb{C}^*$, $(x_1,x_2)\to x_1^2$. This map extends to~$M$, so we get that $K^{-1}_A(M)\cong \mathbb{Z}$ with generator
corresponding to the map $M\to \mathbb{C}^*$, $(x_1,x_2)\mapsto
x_1^2$. Returning to the long exact sequence of the pair $\big(\mathbb{C}^2,M\big)$,
we get that the connecting morphism $K^{-1}_A(M)\to K^0_A\big(\mathbb{C}^2,M\big)$
maps the generator of $K^{-1}_A(M)$ to the complex
$\xymatrix{\underline{\mathbb{C}}\ar[r]^-{x_1^2}& \underline{\mathbb{C}}}$. This
completes the proof of the proposition.
\end{proof}

\subsection{Proof of Lemma \ref{le:fermat-tor}}\label{sec:fermat-tor}

We follow the same strategy as in the proof of Proposition~\ref{le:rel_K-D}. Let us denote by $M\subset \mathbb{C}^2$ the Riemann
surface defined by $x_1^{a_1}+x_2^{a_2}=1$. Let
$A=\big\{t\in (\mathbb{C}^*)^2\, |\, t_1^{a_1}=t_2^{a_2}=1\big\}$. Since
$K^{-1}_{G^T}\big(V^T_1\big)\cong K^{-1}_A(M)$, we get that it is
sufficient to prove that $K^{-1}_A(M)$ is torsion free. Let
$\pi\colon M\to \mathbb{C}$ be the map $(x_1,x_2)\mapsto x_1^{a_1}$. The map $\pi$
is a branched covering with only one branching point, that is, $1\in
\mathbb{C}$. The corresponding ramification points are $R=\big\{(\xi,0)\, |\,
\xi^{a_1}=1\big\}$. The torsion freeness can be deduced easily from the
long exact sequence of the pair $(M,M\setminus{R})$. The action of $A$
on $M\setminus{R}$ is free, so we have
\begin{gather*}
K^{-1}_A(M\setminus{R}) = K^{-1}((M\setminus{R})/A) =
K^{-1}(\mathbb{C}\setminus{\{1\}})\cong \mathbb{Z}.
\end{gather*}
Using the Thom isomorphism for the normal bundle to $R$ in $M$, we get
$K^{-1}_A(M,M\setminus{R}) = K^{-1}_A(R)$. On the other hand, note that $R$ is
the orbit of $A$ through the point $(1,0)\in M$, we get $R=A/B$, where
$B\subset A$ is the cyclic subgroup generated by $(1,\eta_{a_2})$,
$\eta_{a_2} = e^{2\pi\mathbf{i}/a_2}$. Therefore,
$K^{-1}_A(R)=K^{-1}_A(A/B)=K^{-1}(B) = 0$. Recalling the long exact sequence of
the pair $(M,M\setminus{R})$, we get
\begin{gather}\label{les}
\xymatrix{
 0\ar[r]&
 K^{-1}_A(M)\ar[r] &
 K^{-1}_A(M\setminus{R})\ar[r]^-{\delta} &
 K^0_A(M,M\setminus{R}).}
\end{gather}
We get that $K^{-1}_A(M)$ can be embedded as a subgroup of
$K^{-1}_A(M\setminus{R})\cong \mathbb{Z}$. The latter is torsion free, so
$K^{-1}_A(M)$ must be also torsion free.
\begin{Remark}
The above argument can be continued to give a direct proof of the
fact that $K^{-1}_A(M)=0$. Namely,
using the Thom isomorphism, we can prove that the group
$K^0_A(M,M\setminus{R})$ is a free Abelian group of rank $a_2$ and
that the complexes $\xymatrix{\big[L_2^{i-1}\ar[r]^-{x_2} & L_2^i\big]}$
 $(1\leq i\leq a_2)$ represent a $\mathbb{Z}$-basis. Moreover, the image of
 the connecting morphism $\delta$ in~\eqref{les} can be computed
 explicitly as well, that is, it coincides with the sum of the above
 complexes. In particular, we get that $\delta$ is an injective map,
 and hence $K^{-1}_A(M)=0$.
\end{Remark}
\begin{Remark}The long exact sequences of the pairs $(M,M\setminus{R})$ and
$(\mathbb{C}^2,M)$ can be computed explicitly, that is, both the groups and
the differentials can be determined. This allows us to give an
alternative proof of formula~\eqref{fermat-K}. We leave
the details to the interested reader.
\end{Remark}

\subsection[The relative K-ring for $E_7$-singularity]{The relative K-ring for $\boldsymbol{E_7}$-singularity}
The argument from the previous section works also for
$E_7$-singularity. Let us only state the result. The proof is
completely analogous.

We have $f^T(x)=x_1^3x_2+x_2^3+x_3^2$.
The group of diagonal symmetries of $f^T$ is
\begin{gather*}
G^T=\big\{t\in (\mathbb{C}^*)^3\, |\, t_1^3t_2=t_2^3 = t_3^2=1\big\}.
\end{gather*}
Let $L_i=\mathbb{C}^3\times \mathbb{C}$ be the $G^T$-equivariant line bundle for
which the action of $G^T$ on $\mathbb{C}$ is given by the character $G^T\to
\mathbb{C}^*$, $(t_1,t_2,t_3)\mapsto t_i$. Let us introduce the following~$7$
complexes of $G^T$-equivariant vector bundles on~$\mathbb{C}^3$:
\begin{gather*}
E_i^{\bullet}\colon \
\xymatrix{
L_1^{i-1} L_3^{-1} \ar[r]^-{d_0} &
L_1^{i-1} \oplus L_1^{i-1} L_2 L_3^{-1} \ar[r]^-{d_1} &
L_1^{i-1}}L_2,\qquad 1\leq i\leq 6,
\end{gather*}
where the differentials are defined by
$d_0(x,\lambda) = (x,-x_3\lambda, x_2 \lambda)$ and
$d_1(x,\lambda_2,\lambda_3) = (x,x_2 \lambda_2 + x_3\lambda_3)$,
\begin{gather*}
E_7^{\bullet}\colon \
\xymatrix{
L_3\ar[r]^-{d_0}&
\underline{\mathbb{C}}\oplus L_3 \ar[r]^-{d_1} &
\underline{\mathbb{C}},
}
\end{gather*}
where
$d_0(x,\lambda) = \big(x,-x_3\lambda, x_2^3 \lambda\big)$ and
$d_1(x,\lambda_2,\lambda_3) = \big(x,x_2^3 \lambda_2 + x_3\lambda_3\big)$.
\begin{Proposition}
Let $V^T=\big\{x\in \mathbb{C}^3\, |\, f^T(x)=1\big\}$. The relative $K$-ring
$K^0_{G^T}\big(\mathbb{C}^3,V^T\big)\cong \mathbb{Z}^7$ and the complexes
$E_i^\bullet$ $(1\leq i\leq 7)$ represent a $\mathbb{Z}$-basis.
\end{Proposition}

\subsection[$\Gamma$-integral structure for $D_N$-singularity]{$\boldsymbol{\Gamma}$-integral structure for $\boldsymbol{D_N}$-singularity}

Let us compute $\operatorname{ch}_\Gamma(E_i^\bullet)$ for $1\leq i\leq
N$. After a straightforward computation we get
that the relative cohomology group
$H\big(\operatorname{Fix}_g\big(\mathbb{C}^3\big), \operatorname{Fix}_g\big(V^T\big)\big)^{G^T}$ is not
zero only in the following two cases: 1) $g=(g_1,g_2,g_3)$ with
$g_i\neq 1$ for all $i$ and 2) $g=(1,1,-1)$.
For the first case, there are $N-1$ elements, that is, $g=\big({-}1,\eta^{2i-1},-1\big)$
$(1\leq i\leq N-1)$ and the fixed point subsets are $ \operatorname{Fix}_g\big(\mathbb{C}^3\big)=\{0\}$ and
$\operatorname{Fix}_g\big(V^T\big)=\varnothing$. Therefore,
$H\big(\operatorname{Fix}_g\big(\mathbb{C}^3\big), \operatorname{Fix}_g\big(V^T\big);\mathbb{C}\big)^{G^T}\cong \mathbb{C}$
is non-trivial only in degree 0 and we denote by $e_{i}:=1$ the
unit of the cohomology group. For the second case,
$ \operatorname{Fix}_g\big(\mathbb{C}^3\big)=\mathbb{C}^2$ and
$\operatorname{Fix}_g\big(V^T\big) = M=\big\{x_1^2+x_1x_2^{N-1}=1\big\}$. The relative
cohomology group $H^i\big(\mathbb{C}^2,M;\mathbb{C}\big)^{G^T}\cong H^{i-1}\big(M/G^{T};\mathbb{C}\big)$
for $i>0$ and $=0$ for $i=0$. As we already explained above
$M/G^T=\mathbb{C}^*$, so the relative cohomology is non-zero only in degree~2, i.e., for $i=2$. Since $M$ is a Stein manifold, we can describe the
relative cohomology in terms of the holomorphic de Rham complexes on $\mathbb{C}^2$ and~$M$. Namely, consider the complex of Abelian groups
\begin{gather}\label{comp}
\Gamma\big(\mathbb{C}^2, \Omega^\bullet_{\mathbb{C}^2}\big)^{G^T} \oplus
\Gamma\big(M,\Omega^{\bullet-1}_M \big)^{G^T},\qquad
d(\omega,\alpha) = (d\omega, \omega|_M-d\alpha).
\end{gather}
A closed form $(\omega,\alpha)$ in degree $i$, that is, $d(\omega,\alpha)=0$,
defines naturally a linear functional on the space of dimension $i$ relative
chains $\gamma\subset \mathbb{C}^2$ with $\partial
\gamma\subset M$, that is,
\begin{gather*}
\gamma\mapsto \int_\gamma \omega - \int_{\partial \gamma} \alpha.
\end{gather*}
Using the de Rham theorem for $\mathbb{C}^2$ and $M$, it is easy to prove
that the above map induces an isomorphism between the $i$-th
cohomology of the complex~\eqref{comp} and
$H^i\big(\mathbb{C}^2,M;\mathbb{C}\big)^{G^T}$. Let us denote by $e_N\in
H^2\big(\mathbb{C}^2,M;\mathbb{C}\big)$ the cohomology class corresponding to the form $\big(0,-\tfrac{1}{2\pi\mathbf{i}}
dx_1/x_1\big)$.

Suppose now that $g=\big({-}1,\eta^{2a-1},-1\big)$, $1\leq a\leq N-1$. Let us compute the
component of $\operatorname{ch}_\Gamma(E_i^\bullet)$ for $1\leq i\leq N-1$ in
$H^0\big(\operatorname{Fix}_g\big(\mathbb{C}^3\big),
\operatorname{Fix}_g\big(V^T\big);\mathbb{C}\big)^{G^T}$.
Note that in this case we have an isomorphism
$K^0\big(\operatorname{Fix}_g\big(\mathbb{C}^3\big), \operatorname{Fix}_g\big(V^T\big)\big)\cong
K^0\big(\operatorname{Fix}_g\big(\mathbb{C}^3\big)\big)$. The image of
$\iota^*\operatorname{Tr}(E_i^\bullet)$ is
\begin{gather*}
\eta^{-(2a-1)(i-1)} L_2^{i-1}(L_1 L_3 -(-L_3-L_1) + \mathbb{C})=
4 \eta^{-(2a-1)(i-1)} \mathbb{C},
\end{gather*}
where again we abused the notation by denoting by $L_i$ the
restriction of $L_i$ to $\operatorname{Fix}_g\big(\mathbb{C}^3\big)=\{0\}$. The
component of the $\Gamma$-class is
\begin{gather*}
\Gamma(L_1+L_2+L_3)=\Gamma(1-1/2)\Gamma(1-(2a-1)/h) \Gamma(1-1/2) e_a,
\end{gather*}
where $h:=2N-2$.
Therefore, the component of $\operatorname{ch}_\Gamma(E_i^\bullet)$ is
\begin{gather*}
2\eta^{-(2a-1)(i-1)} \Gamma(1-m_a/h) e_a.
\end{gather*}
The component of $\operatorname{ch}_\Gamma(E_N)$ is clearly 0, because
the image of the complex $E_N^\bullet$ in
$K^0\big(\operatorname{Fix}_g\big(\mathbb{C}^3\big)\big)$ is 0.

Suppose that $g=(1,1,-1)$. Let us compute the component of
$\operatorname{ch}_\Gamma(E_i^\bullet)$ in
\[
H^2\big(\operatorname{Fix}_g\big(\mathbb{C}^3\big),
\operatorname{Fix}_g\big(V^T\big);\mathbb{C}\big)^{G^T}=
H^2\big(\mathbb{C}^2,M\big)^{G^T}=\mathbb{C} e_N.
\]
By definition the component of
$\iota^*\operatorname{Tr} (E_i^\bullet)$ in $K^0\big(\mathbb{C}^2,M\big)^{G^T}$ is
\begin{gather*}
-L_2^{i-1}
\xymatrix{\big[
 L_1L_3\ar[r]^-{-x_1} & L_3\ar[r] & 0\big]} +
L_2^{i-1}
\xymatrix{\big[ 0 \ar[r] &
 L_1 \ar[r]^-{x_1} & \underline{\mathbb{C}}\big], }
\end{gather*}
where the above complexes are concentrated in degrees $0$, $1$, and
$2$ and the vector bundles $L_i$ ($1\leq i\leq 3$) are trivial line bundles on
$\mathbb{C}^2$. The second complex, as an element of $K^0\big(\mathbb{C}^2,M\big)$, is
equivalent to the two-term complex
$\xymatrix{\underline{\mathbb{C}}\ar[r]^-{\overline{x}_1} & \underline{\mathbb{C}}}$.
Therefore, the component of $\iota^*\operatorname{Tr} (E_i^\bullet)$
($1\leq i\leq N-1$) takes the form
\begin{gather*}
\xymatrix{-\big[\underline{\mathbb{C}}\ar[r]^-{x_1} & \underline{\mathbb{C}}\big]}+
\xymatrix{\big[\underline{\mathbb{C}}\ar[r]^-{\overline{x}_1} & \underline{\mathbb{C}}\big].}
\end{gather*}
In order to compute the Chern character of the above complexes, we use
the following commutative diagram:
\begin{gather*}
\xymatrix{
 \widetilde{K}^0(\Sigma M) \ar[d]_{\operatorname{ch}} \cong
 \widetilde{K}^{-1}(M) \ar[r]^-{\cong} &
 K^0\big(\mathbb{C}^2,M\big) \ar[d]^{\operatorname{ch}} \\
 H^2(\Sigma M) \cong H^1(M) \ar[r]^-{\delta} & H^2\big(\mathbb{C}^2,M\big), }
\end{gather*}
where the horizontal arrows come from the long exact sequence of the
pair $\big(\mathbb{C}^2,M\big)$ and the vertical arrows are isomorphisms. Under the
isomorphism $\widetilde{K}^0(\Sigma M) \cong K^0\big(\mathbb{C}^2,M\big)$, the complex
$\xymatrix{\underline{\mathbb{C}}\ar[r]^-{x_1} & \underline{\mathbb{C}}}$
corresponds to $P-1$, where $P$ is a line bundle on $\Sigma M$ obtained by gluing
two trivial line bundles along $M$ using the gluing function $M\to
\mathbb{C}^*$, $(x_1,x_2)\mapsto x_1$. The first Chern class of $P$ is easy
to compute. If $\gamma$ is a closed loop in $M$ representing a
cohomology class in~$H_1(M)$, then $\Sigma \gamma$ is a sphere in
$H_2(\Sigma M)$ and hence $P|_{\Sigma \gamma}$ is a line bundle on the
sphere obtained from gluing two trivial line bundles on the two
hemi-spheres along the equator $\gamma$ using the map $\gamma\to
\mathbb{C}^*$, $(x_1,x_2)\to x_1$. By definition $\langle c_1(P), \Sigma
\gamma\rangle$ coincides with the degree of the map $\gamma\to
\mathbbm{S}^1$, $(x_1,x_2)\mapsto x_1/|x_1|$, that is,
\begin{gather*}
\langle c_1(P), \Sigma \gamma\rangle =
\frac{1}{2\pi\mathbf{i}} \int_\gamma \frac{dx_1}{x_1}.
\end{gather*}
In other words, under the suspension isomorphism, $c_1(P)$ coincides
with the de Rham cohomo\-logy class of the form $\tfrac{1}{2\pi\mathbf{i}}
dx_1/x_1$. Recalling the de Rham model for the relative
cohomology group $H^2\big(\mathbb{C}^2,M\big)$, we get that $\delta(c_1(P)) =
e_N$. Note that $c_1(P)=\operatorname{ch}(P-1)$, so
$\operatorname{ch}\big(\xymatrix{\underline{\mathbb{C}}\ar[r]^-{x_1} &
 \underline{\mathbb{C}}} \big) = e_N$. The vector bundle corresponding to the
other complex
$\xymatrix{\underline{\mathbb{C}}\ar[r]^-{\overline{x}_1} & \underline{\mathbb{C}}}$
is $P^{-1}$, so we get
\begin{gather*}
\operatorname{ch}(\iota^*\operatorname{Tr}(E_i^\bullet)) = - 2e_N.
\end{gather*}
Hence
\begin{gather*}
\operatorname{ch}_\Gamma(E_i^\bullet)_g = \frac{1}{2\pi} \Gamma(1/2)
(2\pi\mathbf{i}) (-2e_N) = -2 \mathbf{i} \Gamma(1/2) e_N,
\end{gather*}
where the index $g$ is to remind us that this is the
component corresponding to the fixed point set of $g=(1,1,-1)$.
The computation of $\operatorname{ch}_\Gamma(E_N^\bullet)$ is the
same, except that everywhere we have to replace the vector bundle $P$
with $P^2$, so
\begin{gather*}
\operatorname{ch}_\Gamma(E_N^\bullet)_g = \frac{1}{2\pi} \Gamma(1/2)
(2\pi\mathbf{i}) (-4e_N) = -4\mathbf{i} \Gamma(1/2) e_N.
\end{gather*}
Combining our computations we get the following result:
\begin{gather*}
\operatorname{ch}_\Gamma(E_i^\bullet) =
2\sum_{a=1}^{N-1} \eta^{-(2a-1)(i-1)} \Gamma(1-m_a/h) e_a -
2 \mathbf{i} \Gamma(1/2) e_N
\end{gather*}
and $\operatorname{ch}_\Gamma(E_N^\bullet)\! =\!-4\mathbf{i} \Gamma(1/2)
e_N$. Comparing with formula~\eqref{D-psi} we get that if we define
\mbox{$\operatorname{mir}(\phi_i)\!=\!e_i$} for $1\leq i\leq N-1$ and
$\operatorname{mir}(\phi_N)=2 e_N$, then the
statement of Theorem~\ref{t1} will hold. The vanishing cycle
$\alpha_k$ ($1\leq k\leq N-1$) corresponds to the relative
$K$-theoretic class of the complex~$E^\bullet_k$.

\subsection[$\Gamma$-integral structure for $E_7$-singularity]{$\boldsymbol{\Gamma}$-integral structure for $\boldsymbol{E_7}$-singularity}
The computation in this case is similar to the case of
$D_N$-singularity. Let us sketch only the main steps and leave the
details as an exercise. The goal is to compute
$\operatorname{ch}_\Gamma(E_l^\bullet)$ for $1\leq l\leq 7$.
After a straightforward computation we get that the relative cohomology group
$H\big(\operatorname{Fix}_g\big(\mathbb{C}^3\big), \operatorname{Fix}_g\big(V^T\big)\big)^{G^T}$ is not
zero only in the following two cases:
1) $g=(g_1,g_2,g_3)$ with $g_i\neq 1$ for all $i$ and
2) $g=(1,1,-1)$.
Put $\eta=e^{2\pi\mathbf{i}/9}$ and $\eta_3=e^{2\pi\mathbf{i}/3}$.
For the first case, there are $6$ elements, that is, $g=\big(\eta^{3i-r},\eta_3^r,-1\big)$
$(1\leq i\leq 3$, $1\leq r\leq 2)$ and the fixed-point subsets are
$ \operatorname{Fix}_g\big(\mathbb{C}^3\big)=\{0\}$ and
$\operatorname{Fix}_g\big(V^T\big)=\varnothing$. Therefore,
$H\big(\operatorname{Fix}_g\big(\mathbb{C}^3\big), \operatorname{Fix}_g\big(V^T\big);\mathbb{C}\big)^{G^T}\cong \mathbb{C}$
is non-trivial only in degree 0 and we denote by
$e_{3i-r}\in H\big(\operatorname{Fix}_g\big(\mathbb{C}^3\big), \operatorname{Fix}_g\big(V^T\big);\mathbb{C}\big)^{G^T}$ the
unit of the cohomology group. For the second case,
$ \operatorname{Fix}_g\big(\mathbb{C}^3\big)=\mathbb{C}^2$ and
$\operatorname{Fix}_g\big(V^T\big) = M=\big\{x_1^3x_2+x_2^3=1\big\}$. The relative
cohomology group $H^i\big(\mathbb{C}^2,M;\mathbb{C}\big)^{G^T}\cong H^{i-1}\big(M/G^{T};\mathbb{C}\big)$
for $i>0$ and $=0$ for $i=0$. Just like in the $D_N$-case, we have
$M/G^T=\mathbb{C}^*$, so the relative cohomology is non-zero only in degree~2, i.e., for $i=2$. Let us denote by $e_7\in H^2\big(\mathbb{C}^2,M\big)^{G^T}$ the
cohomology class corresponding to the differential form
$\big(0,-\tfrac{1}{2\pi i} dx_2/x_2\big)$.

Suppose that $g=\big(\eta^{3i-r},\eta_3^r,-1\big)$. We have
\begin{gather*}
\iota^*\operatorname{Tr}(E^\bullet_l)_g =
-2\big(1-\eta_3^{-r}\big)\eta^{-(3i-r)(l-1)}\mathbb{C} \in K^0\big(\operatorname{Fix}_g\big(\mathbb{C}^3\big)\big)
\end{gather*}
and the component of the $\Gamma$-class of $\big[T\mathbb{C}^3/G^T\big]$ in
$H\big(\operatorname{Fix}_g\big(\mathbb{C}^3\big)/G^T\big)$ is
\begin{gather*}
\Gamma(1-(3i-r)/9) \Gamma(1-r/3) \Gamma(1/2).
\end{gather*}
Therefore
\begin{gather*}
\operatorname{ch}_\Gamma(E^\bullet_l)_g =
-\frac{1}{\sqrt{\pi}}
\big(1-\eta_3^{-r}\big)\eta^{-(3i-r)(l-1)}
\Gamma(1-(3i-r)/9) \Gamma(1-r/3) e_{3i-r},\qquad 1\leq l\leq 6.
\end{gather*}
Note that $\operatorname{ch}_\Gamma(E^\bullet_7)_g=0$.

Suppose now that $g=(1,1,-1)$. Then we have
\begin{gather}\label{broad-sector}
\iota^*\operatorname{Tr}(E^\bullet_l)_g =
\xymatrix{-\big[
 \underline{\mathbb{C}}\ar[r]^-{x_2} & \underline{\mathbb{C}} \ar[r] & 0\big]} +
\xymatrix{\big[
 0\ar[r] &
 \underline{\mathbb{C}}\ar[r]^-{x_2} & \underline{\mathbb{C}}\big].}
\end{gather}
Under the isomorphism $\widetilde{K}^0(\Sigma M)\cong K^0\big(\mathbb{C}^2,M\big)$,
the complex $\big[
\xymatrix{
 \underline{\mathbb{C}}\ar[r]^-{x_2} & \underline{\mathbb{C}}} \big]$
corresponds to $P-1$, where $P$ is a vector bundle on $\Sigma M$ obtained by gluing
two trivial line bundles along $M$ with gluing function $M\to \mathbb{C}^*$,
$(x_1,x_2)\mapsto x_2$. Under the suspension isomorphism $H^2(\Sigma
M)\cong H^1(M)$, we have that $\operatorname{ch}(P-1) = c_1(P) $ is
the cohomology class corresponding to the form $\tfrac{1}{2\pi
 \mathbf{i}}dx_2/x_2$. The latter, under the boundary isomorphism $H^1(M)\to
H^2(\mathbb{C}^2,M)$ is mapped precisely to $e_7$, that is,
$\operatorname{ch}\big(\big[
\xymatrix{
\underline{\mathbb{C}}\ar[r]^-{x_2} & \underline{\mathbb{C}}}\big]\big) =
e_7$. The Chern character of the second
complex in~\eqref{broad-sector} is $-e_7$, so we get
$\operatorname{ch}(\iota^*\operatorname{Tr}(E^\bullet_l))_g=-2e_7$. Hence
\begin{gather*}
\operatorname{ch}_\Gamma(E_l^\bullet)_g = \frac{1}{2\pi} \Gamma(1/2)
(2\pi\mathbf{i}) (-2 e_7) = -2\mathbf{i} \sqrt{\pi} e_7.
\end{gather*}
The computation of $\operatorname{ch}_\Gamma(E_7^\bullet)_g $ is the
same as above except that we have to replace the bundle~$P$ with~$P^3$, that is, we get
$\operatorname{ch}_\Gamma(E_7^\bullet)_g =-6\mathbf{i} \sqrt{\pi} e_7.$

Collecting the results of our computations we get
\begin{gather*}
\operatorname{ch}_\Gamma(E_l^\bullet) = -\sum_{r=1}^2\sum_{i=1}^3
\frac{1}{\sqrt{\pi}}
\big(1-\eta_3^{-r}\big)\eta^{-(3i-r)(l-1)}
\Gamma(1-(3i-r)/9) \Gamma(1-r/3) e_{3i-r} -
2\mathbf{i} \sqrt{\pi} e_7
\end{gather*}
for $1\leq l\leq 6$ and $\operatorname{ch}_\Gamma(E_7^\bullet) =-6\mathbf{i} \sqrt{\pi} e_7.$ Let
us compare the above formula with~\eqref{E7-psi}. Note that $1-\eta^{-1}_3
= \sqrt{3}e^{\pi \mathbf{i}/6}$ and $1-\eta_3^{-2}= \sqrt{3}e^{-\pi\mathbf{i}/6}$. We
would like to find $k,a\in \mathbb{Z}_3$, such that,
$\operatorname{mir}\circ \Psi(\alpha_{k,a}) =
\operatorname{ch}_\Gamma(E_l^\bullet) $. Let us write $l-1=3m+k$ for
$0\leq k\leq 2$, $0\leq m\leq 1$. Then the above formula will hold if
we choose $a=-m$ and define
\begin{gather*}
\operatorname{mir}(\phi_i)=e_{3i-1},\quad 1\leq i\leq 3,\qquad
\operatorname{mir}(\phi_{j+3}) = e_{3j-2}, \quad 1\leq j\leq 3,\qquad
\operatorname{mir}(\phi_7) = -e_7.
\end{gather*}
Note that
$\operatorname{mir}\circ \Psi(\alpha_{k,0}+\alpha_{k,1}+\alpha_{k,2}) =
\operatorname{ch}_\Gamma(E_7^\bullet) $. Using these formulas, we get
immediately that the maps $\operatorname{mir}\circ \Psi$ and
$\operatorname{ch}_\Gamma$ identify the Milnor lattice
$H_2\big(f^{-1}(1);\mathbb{Z}\big)$ with the relative K-ring $K^0\big(\mathbb{C}^3, V^T\big)$. This
completes the proof of Theorem~\ref{t1}.

\subsection*{Acknowledgements}

The work of T.M.\ is partially supported by JSPS Grant-in-Aid (Kiban C) 17K05193 and by the World Premier International Research Center Initiative (WPI Initiative), MEXT, Japan. The work of C.Z.\ is supported by JSPS DC1 program including Grant-in-Aid and was previously supported by MEXT scholarship. We would like to thank the referees to our paper for pointing out several inaccuracies and useful suggestions.

\pdfbookmark[1]{References}{ref}
\LastPageEnding


\begin{thebibliography}{99}
\footnotesize\itemsep=0pt

\bibitem{MR966191}
Arnol'd V.I., Guse\u{\i}n-Zade S.M., Varchenko A.N., Singularities of
 differentiable maps, {V}ol.~{II}. Monodromy and asymptotics of integrals,
 \textit{Monographs in Mathematics}, Vol.~83, \href{https://doi.org/10.1007/978-1-4612-3940-6}{Birkh\"auser Boston, Inc.},
 Boston, MA, 1988.

\bibitem{MR1076708}
Atiyah M., Segal G., On equivariant {E}uler characteristics, \href{https://doi.org/10.1016/0393-0440(89)90032-6}{\textit{J.~Geom.
 Phys.}} \textbf{6} (1989), 671--677.

\bibitem{MR1214325}
Berglund P., H\"ubsch T., A generalized construction of mirror manifolds,
 \href{https://doi.org/10.1016/0550-3213(93)90250-S}{\textit{Nuclear Phys.~B}} \textbf{393} (1993), 377--391,
 \href{https://arxiv.org/abs/hep-th/9201014}{arXiv:hep-th/9201014}.

\bibitem{MR3210178}
Chiodo A., Iritani H., Ruan Y., Landau--{G}inzburg/{C}alabi--{Y}au
 correspondence, global mirror symmetry and {O}rlov equivalence, \href{https://doi.org/10.1007/s10240-013-0056-z}{\textit{Publ.
 Math. Inst. Hautes \'Etudes Sci.}} \textbf{119} (2014), 127--216,
 \href{https://arxiv.org/abs/1201.0813}{arXiv:1201.0813}.

\bibitem{MR3888782}
Chiodo A., Nagel J., The hybrid {L}andau--{G}inzburg models of {C}alabi--{Y}au
 complete intersections, in Topological Recursion and its Influence in
 Analysis, Geometry, and Topology, \textit{Proc. Sympos. Pure Math.}, Vol.~100, Amer. Math. Soc., Providence, RI, 2018, 103--117, \href{https://arxiv.org/abs/1506.02989}{arXiv:1506.02989}.

\bibitem{MR3043578}
Fan H., Jarvis T., Ruan Y., The {W}itten equation, mirror symmetry, and quantum
 singularity theory, \href{https://doi.org/10.4007/annals.2013.178.1.1}{\textit{Ann. of Math.}} \textbf{178} (2013), 1--106,
 \href{https://arxiv.org/abs/0712.4021}{arXiv:0712.4021}.

\bibitem{MR2734560}
Frenkel E., Givental A., Milanov T., Soliton equations, vertex operators, and
 simple singularities, \href{https://doi.org/10.1007/s11853-010-0035-6}{\textit{Funct. Anal. Other Math.}} \textbf{3} (2010),
 47--63, \href{https://arxiv.org/abs/0909.4032}{arXiv:0909.4032}.

\bibitem{MR1901075}
Givental A.B., Gromov--{W}itten invariants and quantization of quadratic
 {H}amiltonians, \href{https://doi.org/10.17323/1609-4514-2001-1-4-551-568}{\textit{Mosc. Math.~J.}} \textbf{1} (2001), 551--568,
 \href{https://arxiv.org/abs/math.AG/0108100}{arXiv:math.AG/0108100}.

\bibitem{MR1866444}
Givental A.B., Semisimple {F}robenius structures at higher genus, \href{https://doi.org/10.1155/S1073792801000605}{\textit{Int.
 Math. Res. Not.}} \textbf{2001} (2001), 1265--1286, \href{https://arxiv.org/abs/math.AG/0008067}{arXiv:math.AG/0008067}.

\bibitem{MR2103007}
Givental A.B., Milanov T.E., Simple singularities and integrable hierarchies,
 in The Breadth of Symplectic and {P}oisson Geometry, \textit{Progr. Math.},
 Vol.~232, \href{https://doi.org/10.1007/0-8176-4419-9_7}{Birkh\"auser Boston}, Boston, MA, 2005, 173--201,
 \href{https://arxiv.org/abs/math.AG/0307176}{arXiv:math.AG/0307176}.

\bibitem{MR1924259}
Hertling C., Frobenius manifolds and moduli spaces for singularities,
 \textit{Cambridge Tracts in Mathematics}, Vol.~151, \href{https://doi.org/10.1017/CBO9780511543104}{Cambridge University
 Press}, Cambridge, 2002.

\bibitem{MR2553377}
Iritani H., An integral structure in quantum cohomology and mirror symmetry for
 toric orbifolds, \href{https://doi.org/10.1016/j.aim.2009.05.016}{\textit{Adv. Math.}} \textbf{222} (2009), 1016--1079,
 \href{https://arxiv.org/abs/0903.1463}{arXiv:0903.1463}.

\bibitem{MR1171758}
Kontsevich M., Intersection theory on the moduli space of curves and the matrix
 {A}iry function, \href{https://doi.org/10.1007/BF02099526}{\textit{Comm. Math. Phys.}} \textbf{147} (1992), 1--23.

\bibitem{MR2801653}
Krawitz M., F{JRW} rings and {L}andau--{G}inzburg mirror symmetry, Ph.D.~Thesis, {U}niversity of Michigan, 2010.

\bibitem{MR259897}
Minami H., A {K}\"unneth formula for equivariant {$K$}-theory, \textit{Osaka
 Math.~J.} \textbf{6} (1969), 143--146.

\bibitem{Sa}
Saito K., On periods of primitive integrals,~I, {P}reprint, Research Institute
 for Mathematical Sciences, Kyoto University, 1982.

\bibitem{MR79769}
Satake I., On a generalization of the notion of manifold, \href{https://doi.org/10.1073/pnas.42.6.359}{\textit{Proc. Nat.
 Acad. Sci. USA}} \textbf{42} (1956), 359--363.

\bibitem{MR234452}
Segal G., Equivariant {$K$}-theory, \href{https://doi.org/10.1007/BF02684593}{\textit{Inst. Hautes \'Etudes Sci. Publ.
 Math.}} (1968), 129--151.

\bibitem{MR2917177}
Teleman C., The structure of 2{D} semi-simple field theories, \href{https://doi.org/10.1007/s00222-011-0352-5}{\textit{Invent.
 Math.}} \textbf{188} (2012), 525--588, \href{https://arxiv.org/abs/0712.0160}{arXiv:0712.0160}.

\end{thebibliography}
\end{document}